\documentclass[a4paper,10pt]{amsart}

\usepackage{inputenc}
\usepackage{color}
\usepackage{graphicx}
\usepackage{times,mathptm}
\usepackage{amsfonts,amscd,amssymb,amsmath}
\usepackage{xypic}
\usepackage{latexsym}

\newtheorem{proposition}{Proposition}[section]
\newtheorem{lemma}[proposition]{Lemma}
\newtheorem{corollary}[proposition]{Corollary}
\newtheorem{theorem}[proposition]{Theorem}

\theoremstyle{definition}
\newtheorem{definition}[proposition]{Definition}
\newtheorem{example}[proposition]{Example}
\newtheorem{examples}[proposition]{Examples}

\newtheorem{observation}[proposition]{Observation}

\theoremstyle{remark}
\newtheorem{remark}[proposition]{Remark}

\newcommand{\proplabel}[1]{\label{prop:#1}}
\newcommand{\propref}[1]{Proposition~\ref{prop:#1}}

\newcommand{\lemlabel}[1]{\label{lem:#1}}
\newcommand{\lemref}[1]{Lemma~\ref{lem:#1}}

\newcommand{\thelabel}[1]{\label{the:#1}}
\newcommand{\theref}[1]{Theorem~\ref{the:#1}}

\newcommand{\corlabel}[1]{\label{cor:#1}}
\newcommand{\corref}[1]{Corollary~\ref{cor:#1}}

\newcommand{\remlabel}[1]{\label{rem:#1}}

\newcommand{\deflabel}[1]{\label{def:#1}}
\newcommand{\defref}[1]{Definition~\ref{def:#1}}

\newcommand{\exalabel}[1]{\label{ex:#1}}
\newcommand{\exaref}[1]{Example~\ref{ex:#1}}

\def\Hom{{\rm Hom}}

\def\Ker{{\rm Ker}}

\def\Im{{\rm Im}}

\def\ot{\otimes}

\def\mc{\mathcal}

\def\-{{\rm-}}

\def\ov{\overline}
\def\un{\underline}

\def\Proj{\rm Proj}
\def\Spec{\rm Spec}

\def\Cov{\rm Cov}

\def\wt{\widetilde}
\def\gen{\rm gen}
\def\cl{\rm cl}
\def\ann{\rm ann}
\def\Ann{\mathcal{A}{\rm nn}}
\def\tors{\rm tors}

\def\mapright#1{\smash{\mathop{\longrightarrow}\limits^{#1}}}

\def\hookmapright#1{\smash{\mathop{\hookrightarrow}\limits^{#1}}}

\makeindex

\begin{document}
\title{Localizations and Sheaves of Glider Representations}
\author[F. Caenepeel]{Frederik Caenepeel}
\address{Department of Mathematics, University of Antwerp, Antwerp, Belgium}
\email{Frederik.Caenepeel@uantwerpen.be}
\author[F. Van Oystaeyen]{Fred Van Oystaeyen}
\address{Department of Mathematics, University of Antwerp, Antwerp, Belgium}
\email{Fred.Vanoystaeyen@uantwerpen.be}
\thanks{The first author is Aspirant PhD Fellow of FWO}

\begin{abstract}
The notion of a glider representation of a chain of normal subgroups of a group is defined by a new structure, i.e. a fragment for a suitable filtration on the group ring. This is a special case of general glider representations defined for a positively filtered ring $R$ with filtration $FR$ and subring $S = F_0R$. Nice examples appear for chains of groups, chains of Lie algebras, rings of differential operators on some variety or $V$-gliders for $W$ for algebraic varieties $V$ and $W$. This paper aims to develop a scheme theory for glider representations via the localizations of filtered modules. With an eye to noncommutative geometry we allow schemes over noncommutative rings with particular attention to so-called almost commutative rings. We consider particular cases of $\Proj~ R$ (e.g. for some P.I. ring $R$) in terms of prime ideals, $R$-tors in terms of torsion theories and $\un{\mc{W}}(R)$ in terms of a noncommutative Grothendieck topology based on words of Ore set localizations.
\end{abstract}
\maketitle

\section{Introduction and Motivation}

A glider representation generalizes the notion of a module. Start from a ring $R$ with a filtration $FR$ given by an ascending chain: $\ldots \subset F_nR \subset F_{n+1}R \subset \ldots$, indexed by the integers $\mathbb{Z}$ and such that $1\in F_0R, F_{n}RF_mR \subset F_{n+m}R$ for every $n,m \in \mathbb{Z}$. We also assume that the filtration is exhaustive, $\cup_n F_nR = R$, and also that it is separated, $\cap_n F_nR = 0$. Then obviously $S = F_0R$ is a subring of $R$ and each $F_nR$ is an $S$-bimodule.\\
A glider representation for $FR$ is an $S$-submodule $M$ of an $R$-module $\Omega$ together with a descending chain $M=M_0 \supset \cdots \supset M_n \supset \cdots$ such that $F_mRM_n \subset M_{n-m}$ for $m\leq n$; so $F_mRM_n \not\subset M_{n-m}$ for $m>n$ but here of course $F_mRM_n \subset \Omega$ always. This definition is a special case of a fragment as defined in \cite{EVO}, \cite{NVo1}, \cite{NVo2} but is more general than a natural fragment also defined in loc. cit.\\
One may generalize parts of module theory to the situation of glider representations but often at the cost of nontrivial modifications, e.g. the theory of irreducibility and semisimplicity. The existence of a filtration $FR$ with $F_0R = S$ is a rather weak link between $S$ and $R$ but in special cases it is enough to obtain some structural theory relating $R$ and $S$ and even the intermediate $F_nR$ connecting $S$ and $R$. Two types of filtrations are motivating us, with many applications in mind. First type is the so called standard filtration induced on $R$ via an epimorphism $K[\un{X}] \to R$ or $K<\un{X}> \to R$, where $K[\un{X}]$ is the ring of polynomials on a set of variables $\un{X}$ and $K<\un{X}>$ is the free $K$-algebra on $\un{X}$. In each case the filtration on $K[\un{X}]$ or $K<\un{X}>$ is obtained by letting $\oplus_\alpha KX_\alpha$ be the part of filtration degree 1 and $F_n = F_1^n$ for $n\geq 0$. So the standard filtrations are positive, well-known examples are the Weyl algebras $A_n(K)$, rings of differential operators on a smooth variety, and also $K[V]$ the coordinate rings of algebraic varieties over a field $K$. For a variety $W$ and a morphism $V \to W$ such that $K[W] \subset K[V]$ we may define a standard filtration,$fK[V]$, with $f_0K[V] = K[W]$ by selecting arbitrary $K[W]$-ring generators for $K[V]$. The Rees ring or blow-up ring $\wt{R}$ may be viewed as a graded subring $\sum_n F_nRX^n \subset R[X]$, the associated graded is $G(R) = \oplus_{n \in \mathbb{Z}} F_nR/F_{n-1}R$. For commutative $R$ we know that $\Proj ~\wt{R}$ may be seen as the projective closure of $\Spec ~R$ by ``glueing'' $\Proj ~G(R)$ to it as the part at infinity, so if $R = K[V]$ then $V^* = \Proj ~ \wt{R}$, $V^* \supset V$ as an open and $V_\infty \subset V^*$ as its closed complement. This is in fact a meta-property that remains valid in the noncommutative situation, after suitably defining all geometric concepts, cf. \cite{VoAG}, \cite{VoWi}.\\
The second motivational type of filtrations is obtained from a tower of subrings $S = R_0 \subset \ldots \subset R_i \subset \ldots \subset R_n = R$, which is a filtration of length $n$ but we view it as a $\mathbb{Z}$-filtration with $F_mR = R$ for $m\geq n$, $F_{-d}R = 0$ for $d > 0$. Cases of special interest are chains of grouprings $K \subset KG_1 \subset \ldots \subset KG_n = R$ for a chain of normalizing subgroups $1 \subset G_1 \subset \ldots \subset G_n =G$, chains of envelopping algebras of Lie algebras or chains of iterated Ore extensions (hence many quantum groups appear as examples)! One very interesting example we have in mind is almost commutative rings, both with filtrations of the first type or of the second given by chains of such rings. Roughly said an almost commutative ring is one where $G(R)$ is commutative, e.g. a PBW-deformation of a commutative ring. Special examples are envelopping algebras and rings of differential operators.\\
In this paper we look at scheme theory of gliders but over the topological space defined over $R$. Since gliders are not $R$-modules this scheme structure is not obvious! For modules there is a theory of structure schemes over a base scheme, the structure sheaf of the ring; in fact this is classical for a commutative ring in Algebraic geometry, but it can also be done for a noncommutative ring in terms of $\Spec~R$, cf. \cite{MVo}, $R$-tors the lattice of torsion theories, cf. \cite{Gol}, or a noncommutative topology, cf. \cite{Vovita} or a noncommutative Grothendieck topology, cf. \cite{VoAG}. So the problem we face here is to construct a structure sheaf of a glider representation over the topological space associated to the ring $R$. The glider contains extra information connected to the structure of the chain used to construct $R$ from $S$ by an iterated construction step (in the filtration). There is algebraic information, e.g. as in the case of glider representations of groups, see \cite{CVo} where we dealt with the Clifford theory. There is also geometric information, not only in the commutative case but also in noncommutative geometry and it is this possibility we want to start investigating here.\\
Now the scheme theory classically depends on localization of modules at torsion theories (cf. \cite{Gol}, \cite{G}, \cite{Ros}) on $R$-mod. Hence, the first step is to study a quotient filtration on filtered rings and modules, the second step is to extend this to glider representations by modifying the localization technique. There are some technical problems if you want the localized filtration of a separated filtration on the ring $R$ to remain separated. We will need a notion of $\kappa$-separatedness for a torsion theory $\kappa$ and this leads to the definition of a so called strong characteristic variety. Classically, for a ring of differential operators $R$ , see \cite{Bj}, (or more general for an almost commutative ring) and a filtered $R$-module $M$, the characteristic variety $\chi(M)$ is defined by the variety $V(\ann G(M))$ where $\ann G(M)$ is the annihilator in $G(R)$ of the graded module $G(M)$. We are interested in a subvariety $\xi(M) \subset \chi(M)$ which contains those prime ideals $P$ such that $G(M)$ is $G(R) - P$-torsion free. For such a prime $P$, the localization at $G(R) - P$ is the stalk at $P$ in the structure sheaf of $G(M)$. We will generalize this to filtered rings with noncommutative associated graded $G(R)$ and define a strong characteristic variety as a closed subset in $G(R)$-tors. We will prove that the separatedness of the quotient filtration at some localization at a torsion theory on $R$-mod, say $\kappa$, follows from torsion freeness of $G(R)$ (also $G(M)$ for a module) at $G\kappa$ some left exact radical associated to $\kappa$ on $G(R)$-gr. Hence we obtain in some sense a noncommutative version of the characteristic variety!\\
In fact we present two related but different approaches, one on the level of filtered $R$-modules, $R$-filt, one on the level of graded modules $\wt{R}$-gr. In the first approach we start from a localization on $R$-mod applied to filtered modules, leading to a scheme theory for gliders $M \subset \Omega$ over the lattice $R$-tors with separability related to the strong characteristic variety $\xi(\Omega)$ as indicated above. The second approach starts from a graded localization at $\wt{\kappa}$ on $\wt{R}$-gr, cf. \cite{Nasvo}, and links it to a localization $\kappa$ on $R$-mod plus a graded localization $G\kappa$ on $G(R)$-gr for those $\wt{\kappa}$ in some affine part of the localizations of $\wt{R}$, i.e. those $\wt{\kappa}$ corresponding well to some $\kappa$ on $R$-mod via dehomogenization, while the complement of this affine part reduces to the correspondence of $\wt{\kappa}$ and some $G(\kappa)$-gr. So in the second approach the scheme theory fits in the philosophy of the projective closure via $\wt{R}$-gr (of the affine part that may be viewed as $R$-tors) and the geometry ``at infinity'' via $G(R)$-tors.\\
The localization of a glider $M \subset \Omega$ may then be obtained as the part of degree zero of the quotient filtration on a localization of $\Omega$, much in the spirit of the construction of a projective scheme (see also \cite{NVo2} for further projective aspects of fragments). A very interesting class of examples is given by localization at Ore sets (geometrically corresponding to affine open subsets!), they allow a scheme theory over a noncommutative Grothendieck topology (cf. \cite{VoAG}, \cite{VoWi}). This requires some facts about consecutive multiple localizations, but since composition of noncommutative localization functors is not a localization functor (see \cite{Vo1}), this forces the consideration of ``localization" at left exact preradicals, adding some extra, though unavoidable technicality. The gain is that a global section theorem a la Serre then holds.\\
After basic facts in Section 2, we present the theory of separated quotient filtrations in Section 3 and the scheme theory in Section 4. The structure sheaves for gliders are indeed obtainable over the topological space defined on the global level ober $R$, even in the noncommutative case, and there is also a version of the global section theorem of J.-P. Serre.\\
There is a lot of work in progress, e.g. chains of Lie algebras, rings of differential operators. In order to give somewhat more motivation for the geometry problems, we look at the end of the paper to a few examples/settings related to classical algebraic geometry, which we also intend to work out further in forthcoming work.

\section{Preliminaries}

Let $R$ be a ring and denote by $R$-mod the category of left $R$-modules. A \emph{preradical} $\rho$ of $R$-mod is a subfunctor of the identity functor. The class of preradicals of $R$-mod is denoted by $\mc{Q}(R)$. A preradical $\rho$ such that $\rho\rho = \rho$ is said to be \emph{idempotent}. A preradical $\rho$ such that $\rho(M/\rho(M)) = 0$  for all $M \in R$-mod, is said to be \emph{radical}.  To a preradical, we associate two classes $(\mc{T}_\rho, \mc{F}_\rho)$ of $R$-modules given by
\begin{eqnarray*}
\mc{T}_\rho &=& \{ M \in R\hbox{-}{\rm mod},~ \rho(M) = M\}~~ (\rm{torsion~class})\\
\mc{F}_\rho &=& \{ M \in R\hbox{-}{\rm mod},~\rho(M) = 0\} ~~~(\rm{torsion~free~class}).
\end{eqnarray*}
A left exact idempotent radical is called a \emph{kernel functor}. In fact, it suffices for a preradical $\rho$ to be left exact and radical in order to be a kernel functor by

\begin{proposition}\cite[Proposition 2.11]{Vovita}\\
For $\rho \in \mc{Q}(R)$, the following are equivalent:
\begin{enumerate}
\item $\rho$ is left exact;
\item For every submodule $N$ of $M$, $\rho(N) = \rho(M) \cap N$;
\item $\rho$ is idempotent and $\mc{T}_\rho$ is closed under subobjects.
\end{enumerate}
\end{proposition}

The pair of the torsion class and the torsion free class for a kernel functor $\kappa$ determines a hereditary torsion theory. Details about general torsion theory and localization can be found in \cite{Gol}, \cite{G}. The localization functor $Q_\kappa: R{\rm-mod} \to R{\rm-mod}$ associates to an $R$-module $M$ its localization $Q_\kappa(M)$. This localization is defined by associating to $\kappa$ a Gabriel filter $\mc{L}(\kappa)$ of left ideals:
$$\mc{L}(\kappa) = \{ L {\rm~left~ideal~of~} R, \kappa(R/L) = R/L\},$$
which satisfies the following four properties from \cite{G}
\begin{enumerate}
\item If $I \in \mc{L}(\kappa)$ and $I \subset J$ for some left ideal $J$, then $J \in \mc{L}(\kappa)$;
\item If $I,\, J \in \mc{L}(\kappa)$, then $I \cap J \in \mc{L}(\kappa)$;
\item If $I \in \mc{L}(\kappa), r \in R$, then $(I:r) = \{ a \in R, ar \in I\} \in \mc{L}(\kappa)$;
\item If $M$ is an $R$-module, then $ m \in \kappa(M)$ if and only if $\exists I \in \mc{L}(\kappa)$ such that $Im = 0$.
\end{enumerate}
The reader should be aware that idempotent in \cite{G} means radical in \cite{Vovita}. In fact, there is a one-to-one correspondence between filters of left ideals $\mc{L}$ satisfying the four properties above and left exact idempotent radicals, i.e. kernel functors! If $\mc{L}$ only satisfies the first three, then one loses the radical property and the localization theory becomes different.\\

We reintroduce fragments over a filtered ring $FR$, with $F_0R = S$. As indicated in the introduction, the original definition of a fragment as given in \cite{NVo1} is slightly confusing. Indeed, elucidating the associativity condition $\bf{f_3}$ given there, rises the question what happens if one works for example with a (finite) algebra filtration. To overcome this confusion, we alter the definition of a fragment over a filtered ring.

\begin{definition}
Let $FR$ be a positive filtration. A (left) $FR$-fragment $M$ is a (left) $S$-module together with a descending chain of subgroups
$$M_0 = M \supset M_1 \supset \cdots \supset M_i \supset \cdots$$
satisfying the following properties\\
$\bf{f_1}$. For every $i \in \mathbb{N}$ there is given an operation of $F_iR$ on $M_i$ by $\varphi_i: F_iR \times M_i \to M,~(\lambda,m) \mapsto \lambda.m$, satisfying $\lambda.(m + n) = \lambda.m + \lambda.n, 1.m = m, (\lambda + \delta).m = \lambda.m + \delta.m$ for $\lambda,\delta \in F_iR$ and $m,n \in M_i$.\\

$\bf{f_2}$. For every $i$ and $j \leq i$ we have a commutative diagram
$$\xymatrix{ M & M_{i-j} \ar@{_{(}->}[l]^i \ar@{^{(}->}[r]_i & M\\
F_iR \times M_i \ar[u]^{\varphi_i} & F_jR \times M_i \ar@{_{(}->}[l]^{i_F} \ar[u] \ar@{^{(}->}[r]_{i_M} & F_jR \times M_j \ar[u]_{\varphi_j}}$$

$\bf{f_3}$. For every $i,j,\mu$ such that $F_iRF_jR \subset F_\mu R$ we have $F_jRM_\mu \subset M_i^* \cap M_{\mu - j}$ in which 
$$M_i^* = \{ m \in M,~F_iRm \subset M\}.$$
Moreover, the following diagram is commutative
$$\xymatrix{
F_iR \times F_jR \times M_\mu \ar[d]_{F_iR \times \varphi_\mu} \ar[rr]^{m \times M_\mu} && F_\mu R \times M_\mu \ar[d]_{\varphi_\mu} \\
F_iR \times M_{\mu - j} \ar[rr]^{\ov{\varphi_i}} && M},$$
in which $\ov{\varphi_i}$ stands for the action of $F_iR$ on $M_i^*$ and $m$ is the multiplication of $R$. Observe that the left vertical arrow is defined, since $1 \in F_0R$ implies that $F_jR \subset F_\mu R$. 
\end{definition}
In the sequel we simply call an $FR$-fragment $M$ a fragment, if no ambiguity on the filtration of the ring exists. All results from \cite{EVO}, \cite{NVo1} and \cite{NVo2} remain valid under this new definition. For convenience of the reader, we try to recall all notions concerning fragments if encountered in this paper.

\begin{example}\exalabel{zeropart}
Let $FR$ be filtration. The degree zero part $F_0R$ becomes an $F^+R$-fragment by putting $(F_0R)_n = F_{-n}R$, where $F^+R$ is given by $F^+_{-n}R = 0, F^+_nR = F_nR, n \geq 0$.
\end{example}

 For a fragment $M$, the star operation gives rise to a new chain $M \supset M_1^* \supset \cdots \supset M_i^* \supset \cdots$, denoted by $M^*$. This $M^*$ is again a fragment because $F_jRM_i^*$ for $j\leq i$ has
$$F_{i-j}RF_jM_i^* \subset F_iRM_i^* \subset M_0^* = M,$$
so $F_jRM_i^* \subset M_{i-j}^*$. If $F_iRF_jR \subset F_\mu R$, then $F_jRM_\mu \subset M_i^*$ holds by definition of the star operation.

\begin{definition}
A fragment $M$ that is contained in an $R$-module $ M \subset \Omega$ such that all operations of $F_iR$ on $M_j$ are induced by the module structure of $\Omega$ are called glider representations or glider fragments. If $M_i = M_i^*$ holds for all $i$, we say that the fragment is natural (in $\Omega$). 
\end{definition}
For glider fragments the condition $f_3$ may be reduced to $F_iRM_\mu \subset M_{\mu - i}$ because all $\phi_i$ are induced from the scalar multiplication of $R$. We may assume for such fragments that $RM = \Omega$. That not every fragment needs to be glider follows from the following example.

\begin{example}
Let $\mathbb{Z} \subset \mathbb{Q} \subset \mathbb{Q} \subset \cdots$ be a filtration of $\mathbb{Q}$. Then 
$$ \mathbb{Q} \times \mathbb{F}_2 \supset \mathbb{Q} \times \{ 0 \} \supset 0 \supset \cdots$$
is a fragment, but $\mathbb{F}_2$ cannot be embedded in a $\mathbb{Q}$-module.
\end{example}

A positive filtration $FA$ with $F_0A = B$ is said to be a standard filtration if $FA$ is finite in the sense that every $F_nA$ is a finitely generated left $B$-module and $F_nA = (F_1A)^n$ for every $n \in \mathbb{N}$, or equivalently: $F_nAF_mA = F_{n+m}A$ for every $n,m \in \mathbb{N}$. 

\begin{examples}
\begin{enumerate}
\item
Let $A$ be a $K$-algebra ($K$ some field) with a positive filtration $FA$ such that $F_0A = K$ and every $F_nA$ a finite dimensional $K$-vectorspace, $F_nA = (F_1A)^n$ for every $n\in \mathbb{N}$. 

\item
$A = K<X_1,\ldots,X_n>/I = K<a_1,\ldots,a_n>$ an affine $K$-algebra, $F_1A = Ka_1 \oplus \cdots \oplus Ka_n$, $F_nA = (F_1A)^n$. Generic example for $I = 0$\\
Every $A,FA$ as in 1. is obtained as an example in 2. because if $F_1A = Ka_1 \oplus \cdots \oplus Ka_n$ then $A = K<a_1,\ldots,a_n>$.
\item Let $fA$ be a strong filtration, i.e. $f_nAf_mA = f_{n+m}A$ for all $n,m \in \mathbb{Z}$. Then $fA$ is finite over $f_0A$. The positive filtration $FA$ given as $f^+A$, $F_nA = f_nA$ for $n \geq 0$ is a standard filtration.
\end{enumerate}
\end{examples}

\begin{definition}
Let $A$ have standard filtration $FA$. An $FA$-fragment $M$ is said to be a standard fragment if and only if $F_nAM_m = M_{m-n}$ for all $n \leq m$, equivalently: for all $n \in \mathbb{N}, F_1AM_n = M_{n-1}$.
\end{definition}
By a morphism $f$ between fragments $M, N$, we mean a $B$-linear map $f: M \to N$ such that $f(M_n) \subset N_n$ and $f(am) = af(m)$ for $a \in F_nA, m \in M_n$ for all $n$. If $f$ is injective, we call $M$ a subfragment of $N$. Moreover, if $M_n = N_n \cap M$ for all $n$, we say $M$ is a strict subfragment. In this case $(N/M)_n = N_n/M_n$ with canonical operations gives a fragment structure on $N/M$. Clearly, $M \subset N$ is strict if and only if $F_nA(N_n \cap M) \subset M$ for all $n$.
\begin{proposition}
If $f: M \to N$ is a fragment morphism then $\Im f$ is a fragment. In case $M$ is standard, so is $\Im f$.
\end{proposition}
\begin{proof}
$(\Im f)_{n-1} = f(M_{n-1}) = f(F_1AM_n) = F_1Rf(M_n) = F_1A(\Im f)_n$.
\end{proof}
\begin{corollary}
If $N$ is a strict subfragment of a standard fragment $M$ then $M/N$ is standard.
\end{corollary}
\begin{proof}
Direct from $(M/N)_n = M_n / N_n$ for every $n$.
\end{proof}

Let  $FA$ be a standard filtration with $F_0A = B$. The graded associated ring $G(A)$ is given by $\oplus_n F_nA/F_{n-1}A$ and for $a \in F_nA \setminus F_{n-1}A$, we denote by $\sigma(a)$ the image of $a$ in $G(A)_n$. Consider an Ore set $S$ in $A$ such that $\sigma(S)$ is an Ore set in $G(A)$ and $S \nsubseteq B $. If $G(A)$ is $\sigma(S)$-torsion free then $A$ is $S$-torsion free; indeed, if for $s \in S$, $sa = 0$, then $\sigma(s)\sigma(a) = 0$ by definition of the product in $G(A)$, hence $\sigma(a) \in t_{\sigma(S)}(G(A))$ and $\sigma(a) \neq 0$.
\begin{lemma}\lemlabel{stronglyfiltered}
In the situation as above: $S^{-1}A$ is strongly filtered, hence $F^+S^{-1}A$ is a standard filtration.
\end{lemma}
\begin{proof}
The quotient filtration $FS^{-1}A$ is defined by putting: 
$$F_nS^{-1}A = \{ x \in S^{-1}A {\rm~such~that~} sx \in A {\rm~for~} s\in A, {\rm~say~} \sigma(s) \in G(A)_p {\rm~entails~} sx \in F_{n+p}A\}.$$
To check whether this is well-defined, look at an $x \in S^{-1}A$ such that $sx \in R$ and $sx \in F_{n+p}A$. Assume there is a $t \in S$, $t \in F_qA \setminus F_{q-1}A, tx \in A$ but assume $tx \notin F_{n+q}A$. By exhaustivity of $FA$, $tx \in \dot{F}_{n+q+\delta}A$ for some $\delta >0$, where $\dot{F}_nA = F_nA \setminus F_{n-1}A$.
Since $S$ is an Ore set, there are $s' \in S, t' \in R$ such that $s't = st'$. Since $G(A)$ is $\sigma(S)$-torsion free, we have for every $s \in S$ that $s\dot{F}_nA \subset \dot{F}_{n +p}A$ where $ p = \deg\sigma(s)$. From $tx \in \dot{F}_{n+q+\delta}A$ we get $s'tx \in  \dot{F}_{n + q + \delta + p'}A$ where $p' = \deg\sigma(s')$. Put $q' = \deg\sigma(t')$. Again by torsion freeness, $p' + q = \deg(s't) = \deg(st') = p + q'$, so $\delta$ must equal zero, a contradiction. We leave it to the reader to check that this indeed gives a separated filtration on $S^{-1}A$. Exhaustivity follows from exhaustivity of $FA$. Clearly, if $s \in F_pA \cap S$, then $s^{-1} \in F_{-p}S^{-1}A$ and $F_nS^{-1}A \cap A = F_nA$ for all $n \in \mathbb{N}$.\\
Consider $x \neq 0$ in $F_nS^{-1}A$, then $sx \in F_{n+p}A$ for some $s \in S, p = \deg\sigma(s)$. Since $F_{n+p}A = (F_1A)^{n+p}$ we may write $sx = \sum_i'a_1^i\cdots a_{n+p}^i$ for $a_j^i \in F_1A$.\\
Thus we have that 
$x = \sum_i's^{-1}(a_1^i\cdots a_{n+p-1}^i)a_{n+p}^i$
 with $s^{-1}(a_1^i\cdots a_{n+p-1}^i)$ in $F_{n-1}S^{-1}A$ and $a_{n+p}^i \in F_1A \subset F_1S^{-1}A$. Hence $F_nS^{-1}A = F_{n-1}S^{-1}AF_1S^{-1}A$. For $n=0$ this yields the strongly graded condition.
\end{proof}

Now consider an Ore set $S$ in $A$ with $S \cap B$ Ore in $A$ (and thus also in $B$) and such that $\sigma(S)$ is Ore in $G(A)$; take for example a saturated Ore set in $A$ over an Ore set intersection with $B$. Assume $G(A)$ is $\sigma(S)$-torsion free and $S \nsubseteq B$. We look at $S \setminus B$. By torsion freeness, $S \setminus B$ is multiplicatively closed.
\begin{lemma}
$S \setminus B$ is an Ore set in $A$.
\end{lemma}
\begin{proof}
Consider $a \in A, s \in S \setminus B$, then since $S$ is Ore in $A$ $\exists s' \in S, a' \in A$ such that $s'a = sa'$. If $s' \in S \setminus B$ then we're done. Assume $s' \in S \cap B$, then (ss')a = (sa')s with $ss' \in S \setminus B$ ($\deg\sigma(ss') > 0$!) and $sa' \in A$.
\end{proof}
\begin{remark}
Since $G(A)$ is $\sigma(S)$-torsion free, $A$ is $S$-torsion free, hence if $as = 0$ for $a \in A,s\in S \setminus B$ then $s'a = 0$ for some $s' \in S$ but then $a = 0$ since $t_S(A) = 0$. So the second Ore condition is trivialy satisfied.
\end{remark}

\begin{proposition}
$S \setminus B$ is an Ore set in $(S\cap B)^{-1}A$.
\end{proposition}
\begin{proof}
Take $s \in S \setminus B, b \in (S\cap B)^{-1}A$, thus for some $s_0 \in S \cap B, s_0b \in A$ (note $A \hookmapright{} (S \cap B)^{-1}A$ since $A$ is $t_{S\cap B}$-torsion free). By the previous lemma, $\exists t \in S \setminus B, a \in A$ such that $ts_0b = as$; since $ts_0 \in S \setminus B$ and $a \in A \subset (S \cap B)^{-1}A$ the Ore condition for $S \setminus B$ in $(S \cap B)^{-1}A$ follows.
\end{proof}
\begin{corollary}
$Q_{S \setminus B}(Q_{S \cap B}(A)) = Q_S(A) = S^{-1}A.$
\end{corollary}
Symmetrically, we have
\begin{proposition}
$S\cap B$ is an Ore set in $(S \setminus B)^{-1}A$.
\end{proposition}
\begin{proof}
Take $b \in (S \setminus B)^{-1}A$, then $\exists t \in S \setminus B, tb \in A$. For given $s \in S \cap B \exists s' \in S \cap B$ such that $s'tx = a's$ with $a' \in A$ ($S \cap B$ Ore in $A$). $s't \in S \setminus B$ since $\deg(s't) = \deg(t) > 0$, hence $(s't)^{-1} \in (S \setminus B)^{-1}A$ and thus $1. x = (s't)^{-1}a's$ with $1 \in S \cap B$ and $(s't)^{-1}a' \in (S \setminus B)^{-1}A$. 

\end{proof}
\begin{corollary}
$Q_{S \cap B}Q_{S \setminus B}(A) = Q_S(A) = Q_{S \setminus B}Q _{S \cap B}(A).$
\end{corollary}
\begin{corollary}
The localizations at $S\cap B$ and $S \setminus B$ are compatible (see \cite{Vo1}).
\end{corollary}

Let $FA$ be the positive part of a strong filtration on $A$ (e.g. when $A$ is $S^{-1}A'$ for some standard filtered $A'$). Consider a standard fragment $M$ over $FA$. 
\begin{proposition}
If $N \subset M$ is a strict sub fragment, then $N$ is standard.
\end{proposition}
\begin{proof}
$N_m = M_m \cap N$ thus $F_1AN_m \subset N_{m-1} = M_{m-1} \cap N = F_1AM_m \cap N.$ If $n \in F_1AM_m \cap N$ then $F_{-1}A n \subset F_{-1}AF_1A M_m = M_m$ and $F_{-1}A n \subset N \cap M_m = N_m$ or $ n \in F_1AN_m$. Thus $F_1AN_m = F_1AM_m \cap N = M_{m-1} \cap N = N_{m-1}$, i.e. $N$ is standard.
\end{proof}
\begin{proposition}
If $FA$ is the positive part of a strong filtration on $A$ and $M$ is a natural $FA$-fragment, then $M$ is standard.
\end{proposition}
\begin{proof}
$M_n = \{ m \in M , F_nRm \subset M \}$, so if $m \in M_{n-1}$ then $F_{-1}A m \subset M_n$ since $F_nAF_{-1}Am = F_{n-1}Am \subset M$. Thus $m \in F_1AM_n \subset M_{n-1}$, i.e. $M$ is standard.
\end{proof}

Sometimes it suffices that $S\cap B$ is Ore in $B$. For example,  if $A$ is an extension of $B$, then
$S\cap B$ Ore in $B$ implies that $S\cap B$ is Ore in $A$. Recall that $A$ is an extension of $B$ if $A = BZ_B(A)$,
where
$$Z_B(A) = \{a\in A,~ ba = ab {\rm~for~all~} b \in B\}.$$
In this case $S\cap B$ Ore in $B$ implies $S\cap B$ Ore in $A$. Indeed, take $a \in A$, $s \in S \cap B$, say $a = \sum b_iz_i$ with $b_i \in B, z_i \in Z_B(A)$. Take $s_i \in S \cap B$ such that $s_ib_i = b_i' s$ for all $i$ (finite sum) and take $s'' \in \cap Bs_i, s'' \in S \cap B$ ($S\cap B$ Ore in $B$ and $1 \in S\cap B$). Then we obtain that $s''b_i = b_i''s$ with $b_i'' \in B$, for all $i$. Hence 
$$s''a = \sum s''b_iz_i = \sum b_i''sz_i = \sum b_i''z_is = (\sum b_i''z_i)s,$$
with $\sum b_i''z_i \in A$, whence $S\cap B$ is Ore in $A$.\\

By the above, we can assume that $S$ is Ore in $A$ such that $S \cap B = 1$ and such that $\sigma(S)$ is an Ore set of $G(A)$, the latter being $\sigma(S)$-torsion free. Now, let $M$ be a standard glider fragment, $M \subset \Omega = AM$. We consider on $\Omega$ the filtration $f\Omega, f_{-n}\Omega = M_n, n \geq 0$ and $f_m\Omega = F_mAM$ for $m \geq 0$. The associated graded $G_f(\Omega)$ has for the negative part $g(M) = \oplus_{i \geq 0}M_i/M_{i+1}$ and this is itself a fragment with respect to the positive part of the grading filtration of $G(A)$ (see \cite{NVo1}).
We assume $G_f(\Omega)$ is $\sigma(S)$-torsion free, thus also elements of $g(M)$ cannot be annihilated by elements of $\sigma(S)$. We also assume (however not necessary) that the body $B(M) = \cap_n M_n = 0$. We define for $d \in \mathbb{Z}$
$$Q_S(\Omega)_d = \{ z \in Q_S(\Omega), sz \in F_{p +d}\Omega, s \in S \cap \dot{F}_pA \}.$$
Just as in the proof of \lemref{stronglyfiltered} one shows that the $d$ is depending only on $z$ and not on the specific $s$. In fact, $Q_S(\Omega) = \cup_d Q_S(\Omega)_d$ is the quotient filtration on $S^{-1}\Omega$. 

\begin{lemma}
Let $X = FX$ be a filtered $FA$-module, with $FA$ strongly filtered. Then $X$ is strongly filtered.
\end{lemma}

\begin{proof}
Look at $F_1AF_{n-1}X \subset F_nX$. Since $F_{-1}AF_nX \subset F_{n-1}$ we have that $F_1AF_{n-1}X \supset F_1AF_{-1}F_nX = F_nX$ or $F_1AF_{n-1}X = F_nX$ for every $n$. From this, one easily deduces that $FX$ is strongly filtered.
\end{proof}

By \lemref{stronglyfiltered} we obtain that $S^{-1}\Omega$ is a strongly filtered and standard filtration. We define the localization $Q_S(M)$ of $M$ at $S$ to be $F_0S^{-1}M$. We have
\begin{proposition}
$Q_S(M)$ with $Q_S(M)_d = F_{-d}S^{-1}\Omega$ is a standard fragment with respect to $F^+S^{-1}A$. Moreover, $Q_S(M) \cap \Omega = M$ and $Q_S(M)_d \cap \Omega = M_d, d \geq 0$.
\end{proposition}
\begin{proof}
The first statement follows from \exaref{zeropart}. If $\omega \in Q_S(M) \cap \Omega$ then $s\omega \in F_p\Omega$ for some $s$ with $\deg\sigma(s) = p$. Since multiplication by $s$ cannot lower the degree of an element of $\Omega$ ($G(\Omega)$ is $\sigma(S)$-torsion free!) it follows that $\omega \in F_0\Omega = M$. Similar, if $\omega \in Q_S(M)_d \cap \Omega$, $s\omega \in F_{p -d}\Omega$, hence $\omega \in F_{-d}\Omega, d \geq 0$, i.e. $w \in M_d$.
\end{proof}

\begin{remark}
In \cite{NVo2}, the authors independently define $Q_S(M)$ and observe afterwards that this corresponds with the degree zero part $F_0S^{-1}\Omega$. If one additionally assumes $M$ to be finitely generated, one can take the completion to obtain $Q_S(M)~\widehat{} = Q^q_S(\Omega)$, the quantum localization of $\Omega$ at $\sigma(S)$. Since localizing at an Ore set is perfect, we get a filtration on $Q^\mu_S(\Omega)$ (see \cite{LiVo}) and $Q_S(M)~\widehat{}~$~becomes a fragment with respect to the positive filtration on this microlocalization.
\end{remark}
\begin{observation}
For $w \in F_0S^{-1}\Omega$ we have $s_p\omega \in F_p\Omega$ for some positive $p$, but with respect to $S^{-1}$ in $S^{-1}A$ we have that for some $s_n$ with $ \deg s_n = n > p$, $s_n^{-1}s_p\omega \in F_{p-n}\Omega = M_{n-p}$. Thus $Q_S(M)/M$ need not be $S$-torsion but it is $S^{-1}$-torsion, since $s_n^{-1}s_p \in S^{-1}$ ($S^{-1}$ is the multiplicative set generated by elements of $S$ and their inverses).\\
Observe moreover that $Q_S(M)$ is natural, since $S^{-1}A$ is strongly filtered. 
\end{observation}

We call $Q_S(M)$ the quotient fragment of $M$ over $S^{-1}A$. In the next section we will discuss how we can localize glider fragments at more general kernel functors. We chose to treat this particular case independently, since many geometrical examples are standard filtered rings with graded associated being a commutative domain (e.g. rings of differential operators with the $\Sigma$-filtration).

\section{Quotient filtrations}

We begin by introducing the quotient filtration on the localization of a filtered ring at a kernel functor $\kappa$ on $R$-mod; this is inherent in \cite{SVo} but there it is applied for saturated localizations because one is mainly interested in lifting a localization from the associated graded ring to arrive at microlocalizations at multiplicatively closed sets $S$ of a filtered ring $FR$. For details on filtered rings we refer to \cite{LiVo} and for generalities on localization one can look at \cite{VerVo}. In general, however, it is not possible to lift the kernel functor $\kappa$ to kernel functors on the Rees or associated graded level. Bearing the (noncommutative) geometry of filtered rings in mind, we therefore develop the theory starting from either some kernel functor on $R$-mod or in a second approach from a kernel functor $\widetilde{\kappa}$ on the Rees ring $\widetilde{R}$ of a filtered ring $FR$, in which case we will obtain a nice relation between the graded and filtered levels.\\

Consider a filtered ring $R$ with Zariskian filtration $FR$. We write $G(R)$ for the associated graded ring and $\widetilde{R}$ for the (graded) Rees ring of $FR$; we use notation and terminology from \cite{LiVo}.\\
A filtration $FM$ on an $R$-module $M$ will always be assumed exhaustive, i.e. $\cup_n F_nM = M$; $FR$ will always be assumed to be separated, i.e. $\cap_n F_nR = 0$.  The associated graded $G(R)$-module of $FM$ will be denoted by $G(M)$, the Rees module by $\widetilde{M}$. Since $FR$ is separated, there exists for homogenous $\widetilde{r} \in \widetilde{R}$ a unique integer $n$ such that $\widetilde{r} = rX^{n(r)}$, where $r \in \dot{F}_nR = F_{n}R - F_{n - 1}R$. We denote this integer by $n = \deg(r)$. Recall: $R{\rm-filt} \to \mc{F}_X, M \mapsto \widetilde{M}$ is an equivalence of categories, where $\mc{F}_X$ denotes the full subcategory of $X$-torsionfree $\widetilde{R}$-modules. Here, $X$ stands for the central element of degree 1 such that $\widetilde{R} = \sum_n F_nRX^n \subset R[X,X^{-1}]$.\\
Let $R$ be a Zariski filtered ring and let $\kappa$ be a localization functor (i.e. a kernel functor) on $R$-mod given by its Gabriel filter $\mc{L}(\kappa)$. One may define 
$$\mc{L}(G(\kappa)) = \{ J {\rm~left~ideal~of~}G(R), J \supset G(L) {\rm~for~some~} L \in \mc{L}(\kappa)\}$$
 and show that this gives rise to a left exact preradical $G(\kappa)$ on $G(R)$-gr in general and to a kernel functor if $G(R)$ is a commutative  domain.
 
\begin{lemma}
$\mc{L}(G(\kappa))$ defines a topology on $G(R)$, i.e. the associated functor $G(\kappa)$ is a kernel functor. If $G(R)$ is a commutatve domain, then $\mc{L}(G(\kappa))$ is a Gabriel filter, i.e. $G(\kappa)$ is an idempotent kernel functor.
\end{lemma}

\begin{proof}
a) If $J \in \mc{L}(G(\kappa))$ and $M \supset J$ then obviously $M \in  \mc{L}(G(\kappa))$.\\
b) If $I, J \in  \mc{L}(G(\kappa))$ then $I \cap J \supset G(H_1) \cap G(H_2) \supset G(H_1 \cap H_2)$ for some $H_1,H_2 \in \mc{L}(\kappa)$ with $G(H_1) \subset I$ and $G(H_2) \subset J$. Thus $I\cap J \in  \mc{L}(G(\kappa))$.\\
c) If $L \in \mc{L}(G(\kappa))$ and $\ov{y} \in G(R)$ then we have to find $ H \in \mc{L}(G(\kappa))$ such that $H\ov{y} \subset L$ and it is sufficient to do this for homogeneous $\ov{y}$ since $G(R)$ is graded, say $y \in \dot{F}_{\deg(y)}R$. Since $L \supset G(I)$ for some $I \in \mc{L}(\kappa)$, there exists an $H \in \mc{L}(\kappa)$ such that $Hy \subset I$. For homogeneous $\ov{h} \in G(H)$, say $h \in \dot{F}_{\deg(h)}R$ is a representative for $\ov{h}$, it holds that $\ov{h}\ov{y} = \ov{hy}\delta_{\deg(hy),\deg(h)+\deg(y)} \in G(R)_{\deg(h)+\deg(y)}$ ($hy \in \dot{F}_{\deg(hy)}R$, with $\deg(yh) \leq \deg(y) + \deg(h)$). Thus $G(H)\ov{y} \subset G(Hy) \subset G(I) \subset L$ with $G(H) \in \mc{L}(G(\kappa))$ as desired.\\
Assume now that $G(R)$ is a domain. Let $L \in \mc{L}(G(\kappa))$ and $H \subset L$ such that $L/H$ is $G(\kappa)$-torsion. We have to establish that $H \in \mc{L}(G(\kappa))$. If $I \in \mc{L}(\kappa)$ is such that $G(I) \subset L$ then for every $\ov{i} \in G(I)$ there is a $G(J_i) \in \mc{L}(G(\kappa))$ such that $G(J_i)\ov{i} \subset H$. Since $FR$ is Zariskian, $G(R)$ is Noetherian, so $G(I)$ is finitely generated, say by $\ov{i_1},\ldots, \ov{i_m}$. Then $J = \cap_{k=1}^m J_{i_k} \in \mc{L}(\kappa)$ satisfies $G(J)G(I) = G(JI) \subset H$ ($G(R)$ commutative domain!) We conclude that $H \in \mc{L}(G(\kappa))$ since the Gabriel filter $\mc{L}(\kappa)$ is closed under multiplication.\end{proof}

Remark that if $G(R)$ is a domain and $\mc{L}(\kappa)$ is symmetric, then the same conclusion holds. Fortunately, we only need the notion of $G(\kappa)$- or $\ov{\kappa}$-torsion, so in fact we can forget about the above lemma and simply define

\begin{definition}
Let $M$ be a graded $G(R)$-module. An element $m \in M$ is called $\ov{\kappa}$-torsion if there exists $I \in \mc{L}(\kappa)$ such that $G(I)m = 0$.  $\ov{\kappa}(M)$ is the set of $\ov{\kappa}$-torsion elements of $M$. If $\ov{\kappa}(M) = 0$, we say that  $M$ is $\ov{\kappa}$-torsion free.
\end{definition}

\begin{lemma}\lemlabel{1}
Let $M$ be a filtered $R$-module. Then $G(\kappa(M)) \subset \ov{\kappa}(G(M))$. This induces an epimorphism 
$$\xymatrix{ G(M/\kappa(M)) \ar@{>>}[r] & G(M)/\ov{\kappa}G(M)}$$
\end{lemma}

\begin{proof}
Let $\ov{n} \in G(\kappa(M))_{\deg(n)}$ and take a representative $n \in \dot{F}_{\deg(n)}\kappa(M)$. There exists $I \in \mc{L}(\kappa)$ such that $In = 0$, hence $G(I)\ov{n} = 0$ in $G(M)$.
\end{proof}
Observe that we do not know whether $G(M)/\ov{\kappa}G(M)$ is $\ov{\kappa}$-torsion free, since we do not know in general whether $\ov{\kappa}$ is radical! In order to define a filtration on the localization, we need the notion of $\kappa$-separatedness.
\begin{definition}
Let $\kappa$ be a kernel functor with associated filter $\mc{L}(\kappa)$ and let $M$ be a filtered $R$-module. If for some $I \in \mc{L}(\kappa)$ such that if $m \in \dot{F}_nM$ with $F_\gamma Im \subset F_{\gamma + n- 1}M$ for all $\gamma$ implies that $m \in \kappa(M)$, then we say that $M$ is $\kappa$-separated.
\end{definition}
In fact, we do not need that $\kappa$ is radical, that is, the definition can also be applied to left exact preradicals. 

\begin{lemma}\lemlabel{2}
If $M/\kappa(M)$ is $\kappa$-separated, then $G(\kappa(M)) = \ov{\kappa}(G(M))$ and $G(M/\kappa(M)) = G(M)/\ov{\kappa}G(M)$.
\end{lemma}
\begin{proof}
Let $\ov{z} \in \ov{\kappa}(G(M))_n$ and let $I \in \mc{L}(\kappa)$ such that $G(I)\ov{z} = 0$. Let $z \in \dot{M}_n$ be a representative for $\ov{z}$. Since $FR$ is Zariskian, the induced filtration on $I$ is good, so there exist $i_1,\ldots,i_k \in I$ and $n_1,\ldots, n_k$ such that $F_nI = \sum_{j=1}^k F_{n - n_j}Ri_j$ for all $n$. Moreover, $\deg(i_j) \leq n_j$ (we even have equality if this is a minimal generating set, since $FR$ is Zariskian). For any $j$, $\ov{i_j}\ov{z} = 0$ in $G(M)_{n + \deg(i_j)}$, i.e. $i_jz \in F_{n + \deg(i_j)- 1}M$. Hence 
$$F_\gamma I\ov{z} = \sum_{j=1}^k F_{\gamma - n_j}R i_j\ov{z} \subset \sum_{j=1}^k F_{\gamma - n_j}R F_{n + \deg(i_j)- 1}M/\kappa(M) \subset F_{\gamma + n- 1}M/\kappa(M).$$
Since $M/\kappa(M)$ is $\kappa$-separated, $z \in \kappa(M)$, or $\ov{z} \in G(\kappa(M))$.
\end{proof}

\begin{example}
In general, we don't have that $G(\kappa(M)) = \ov{\kappa}(G(M))$. Consider for example the ring $R = \mathbb{C}[X,Y]/(XY-1)$ with standard filtration. Then the associated graded ring is $G(R) = \mathbb{C}[x,y]/(xy)$ with $x = \ov{X}, y = \ov{Y}$. Let $S = S_X$ be the multiplicative set $\{1,X,X^2,\ldots\}$. Then $\ov{\kappa_S}$ is given by the multiplicative set $\ov{S} = S_x = \{1,x,x^2,\ldots\}$ in $G(R)$. Since $R$ is a domain, we have that $G(\kappa_S(R)) = G(0) = 0$. But $\kappa_{\ov{S}}(G(R)) = (y)$. 
\end{example}

\begin{lemma}\lemlabel{tf}
If $G(M)$ is $\ov{\kappa}$-torsion free then $M$ is $\kappa$-separated.
\end{lemma}
\begin{proof}
Let $ m \in \dot{F}_nM$, $\ov{m} = \sigma(m) \in G(M)_n$ and assume that for $I \in \mc{L}(\kappa)$ and all $\gamma, F_\gamma I m \subset F_{\gamma + n - 1}M$. Then clearly: $G(I) \ov{m} = 0$. Since $G(M)$ is $\ov{\kappa}$-torsion free, $\ov{m} = 0$, a contradiction, unless $m = 0 \in \kappa(M)$. In fact $\kappa(M) = 0$ because if $Jm = 0$ for some $J \in \mc{L}(\kappa)$ then $J_\gamma m = 0 \subset F_{\gamma + n - 1}M$ for $m \in \dot{F}_nM$ ($m \neq 0$). As above, this leads to a contradiction.
\end{proof}
\begin{corollary}\corlabel{almostcom}
If $G(R)$ is a commutative domain then $R$ is $\kappa$-separated for every kernel functor (or more general, for every preradical) $\kappa$ on $R$-mod.
\end{corollary}

Before we state and prove the theorem on localizations in our first approach, we deduce the existence of a strong characteristic variety, which will turn out to be a closed set in $R$-tors such that the localizations remain separated.\\

For rings of differential operators $R$, see \cite{Bj}, \cite{LiVo}, and a filtered $R$-module $M$ the characteristic variety of $M$ is defined over the commutative ring $G(R)$ by taking $\ann_{G(R)}G(M)= I$ and letting the characteristic variety $\chi(M)$ be given by $V(I)$ in $\Spec~G(R)$. For a prime $P \in \chi(M)$, we have that $G(R) - P \cap I = \emptyset$, but this is a more general statement than to say that $G(M)$ is $\ov{\kappa}_P$-torsion free, the statement we want by \lemref{tf}. So we have to look at a smaller set of prime ideals. Now in the generality considered here, i.e. both $R$ and $G(R)$ being noncommutative, where we may assume $R$ to be positively filtered and $G(R)$ Noetherian (hence $FR$ is a Zariskian filtration), we introduce the \emph{strong characteristic variety} for an arbitrary separated filtered $R$-module, $M$ say, as a subset $V = \xi(M)$ in $G(R)$-tors. The following theorem will yield a separated quotient filtration on the localization of $M$ at $\kappa$ to which $\ov{\kappa} \in \xi(M)$ in $G(R)$-tors is associated.\\

Define $\Ann(M)$ with script $\mc{A}$ as the set
 $$ \Ann(M) = \{ \ov{L},~\ov{L} {\rm~left~ideal~of~} G(R),~\ov{L}\ov{m} = 0 {\rm~for~some~} \ov{m} \in G(M)\}.$$
 Define $\xi(M) = V \subset G(R)$-tors by putting 
 $$V = \{ \ov{\kappa},~ \mc{L}(\ov{\kappa}) \cap \Ann(M) = \emptyset\}.$$
 Thus we have for $\ov{\kappa} \in V$ that $G(M)$ is $\ov{\kappa}$-torsion free, hence $M$ is $\kappa$-separated by \lemref{tf}. Observe that $\ov{\kappa(M)} = \ov{\kappa}(G(M))$ and $G(M / \kappa M) = G(M)/\ov{\kappa}G(M)$ by \lemref{2}. Now $V$ is closed in the gen-topology of $G(R)$-tors (see Section 4 for the definition of the gen-topology). Indeed, for $\ov{\tau} \in G(R)-\tors - V$ we have $\mc{L}(\tau) \cap \Ann(M) \neq \emptyset$ or $\ov{\tau}(G(M)) \neq 0$. If $\ov{\gamma} \geq \ov{\tau}$ then $\ov{\gamma}(G(M)) \neq 0$ too or $\ov{\gamma} \in G(R)-\tors - V$, proving that $\gen(\ov{\tau}) \subset G(R)-\tors - V$, i.e. that $G(R)-\tors - V$ is gen-open, and thus $V$ is gen-closed. If $\kappa \in R-\tors$ is such that $\ov{\kappa} \in V$ (observe that the whole construction works for $\ov{\kappa}$ a left exact preradical, then $V$ is constructed in the noncommutative topology $G(R)$-pretors) then $\kappa \in V_R$ (i.e. those $\tau \in R-\tors$ with $\ov{\tau} \in V \subset G(R)$-pretors). Hence by the following theorem, $Q_\kappa(M)$ has a separated quotient filtration. Moreover, $G(\kappa(M)) = \ov{\kappa}(G(M)) =0$ yields that $\kappa(M) = 0$ as $\kappa(M)$ is separated filtered by $FM \cap K(M)$.\\
 
Going back to the classical setting of rings of differential operators, let $P \in \chi(M) - \xi(M)$ (where we associated $P$ to the torsion theory $\kappa_P$). We still obtain a filtered localization $S^{-1}M$, where $S$ is an Ore set yielding $\ov{S} = G(R) - P$, but the filtration is no longer separated. Indeed, lifts of $\ov{S}$-torsion elements are in the core of the filtration (see proof of theorem). However, $P \in \chi(M)$, so the core of the filtration is not the whole of the localized module! Hence, even in the classical setting of almost commutative rings there are questions of how both characteristic varieties are related. This is work in progress, but we include an example
\begin{example}
Let $A_1 = \mathbb{C}<X,\partial_X>/(X\partial_X - \partial_XX -1)$ be the first Weyl algebra and consider the holonomic module $M = A_1/A_1P$, where $P = X\partial_X$. The sigma-filtration on $A_1$ is good and Zariskian, so we have that the induced filtration on $M$ is good and \newline $G(M) = \mathbb{C}[X,\ov{\partial_X}]/(X\ov{\partial_X})$. Hence the characteristic variety is $\chi(M) = V((X\ov{\partial_X}))$ and in particular $(X) \in \chi(M)$. However, $(X) \notin \xi(M)$, because $\ov{\partial_X} \in G(A_1) - (X)$ kills $X \in G(M)$. So in general, the characteristic variety is strictly bigger than the strong characteristic variety.
\end{example}

\begin{theorem} \thelabel{locfilt}
Let $FM$ be a filtered module over the Zariskian filtered ring $FR$ such that $M/\kappa(M)$ is $\kappa$-separated. The localized module $Q_\kappa(M)$ has a filtration $FQ_\kappa(M)$ making the localization morphism $j_\kappa: M \to Q_\kappa(M)$ into a strict filtered morphism. Similar for $FR$ such that $R/\kappa(R)$ is $\kappa$-separated, $j_\kappa: R \to Q_\kappa(R)$ is a strict filtered morphism and a filtered ring morphism. Moreover, $Q_\kappa(M)$ with the quotient filtration defined here is a filtered $Q_\kappa(R)$-module with respect to the quotient filtration on $Q_\kappa(R)$. 
\end{theorem}

\begin{proof}
We have a strict exact sequence $0 \to \kappa M \to M \to M/\kappa M \to 0$, hence by exactness of $G$ on strict sequences we obtain an exact sequence
$$ 0 \to G(\kappa M) \to G(M) \to G( M/ \kappa M) \to 0$$
of graded $G(R)$-modules. Since $M/\kappa(M)$ is $\kappa$-separated, $G(\kappa M) = \ov{\kappa}G(M)$ and $G(M/\kappa M)$ equals $G(M)/\ov{\kappa}G(M)$. This allows us to reduce the situation to the case where $M$ is $\kappa$-torsion free and the canonical map $j_\kappa: M \to Q_\kappa(M)$ is injective.\\
For $x \neq 0$ in $Q_\kappa(M)$ there is an $I \in \mc{L}(\kappa)$ such that $Ix \subset M$. Since $FR$ is Zariskian, the filtration induced by $FR$ on $I$ is a good filtration, hence for all $n \in \mathbb{Z}, F_nI = \sum_{i=1}^s F_{n - d_i}R\zeta_i$, for $\zeta_1,\ldots,\zeta_s \in I$ with $\deg \sigma(\zeta_i) \leq d_i \in \mathbb{Z}$. Then $\zeta_i x \in F_{d_i + \gamma_i}M$ for some $\gamma_i \in \mathbb{Z}$. Put $\gamma = \max\{\gamma_i, i = 1,\ldots,s\}$, then $Ix \subset M$ and for all $n \in \mathbb{Z}$
\begin{equation} \label{fn}
  F_nIx \subset F_{n+\gamma}M,
  \end{equation}
 because $F_nI x \subset \sum F_{n - d_i}R\zeta_i x \subset F_{n - d_i + d_i + \gamma}M = F_{n +\gamma}M$.\\
Since $FM$ is separated (the original $FM$ was separated, so $FM/\kappa M$ is separated as well) and $Ix \neq 0$ as $Q_\kappa(M)$ is $\kappa$-torsion free, there is a minimal $\gamma$ such that for all $n, F_nIx \subset F_{n+\gamma}M$. For this $\gamma$ we have that there is an $n \in \mathbb{Z}$ such that $F_nI x\subset F_{n+\gamma}M - F_{n + \gamma - 1}M$. Let us check that $\gamma$ depends on $x$ but not on the $J \in \mc{L}(\kappa)$ chosen. Pick $\zeta \in F_nI$, $n$ as just mentioned, such that $\zeta x \in F_{n+\gamma}M - F_{n + \gamma - 1}M$ and now look at $J \in \mc{L}(\kappa)$ such that $Jx \subset M$, with $F_\alpha J x \subset F_{\alpha + \gamma - 1 }M$ for all $\alpha \in \mathbb{Z}$. Then $(J : \zeta) \in \mc{L}(\kappa)$ and for all $m\in \mathbb{Z}$, 
$$F_m(J : \zeta) \zeta x \subset (F_{m + n}J)x \subset F_{m+n + \gamma - 1}M.$$
Since $M$ is $\kappa$-separated, $\zeta x \in \kappa(M) = 0$, a contradiction. So we define
\begin{eqnarray*}
F_\gamma Q_\kappa(M) &=& \{q \in Q_\kappa(M),  \exists I \in \mc{L}(\kappa) {\rm~such~that~} F_nIq \subset F_{n+ \gamma}M {\rm~for~ all~} n \in \mathbb{Z}\}\\
&=& \{ q \in Q_\kappa(M), v(q) \leq \gamma\},
\end{eqnarray*}
where the \emph{filtration degree} $v$ is given by $v(q) = \gamma$ where $\gamma$ is as constructed above. It is obvious that $F_\gamma Q_\kappa(M)$ defines an ascending chain of additive subgroups of $Q_\kappa(M)$. Now first look at $M=R$ and $x,y \in Q_\kappa(R)$ such that $I_1x \subset R, I_2y \subset R$ and both satisfying \eqref{fn}. Put $J = (I_2:x)$, then for all $n \in \mathbb{Z}$: $F_n(J \cap I_1)x \subset F_{n + v(x)}R \cap I_2 = F_{n + v(x)}I_2$, and also
$$F_n(J \cap I_1)xy \subset F_{n+v(x)}I_2y \subset F_{n + v(x) + v(y)}R.$$
Hence $FQ_\kappa(R)$ makes $Q_\kappa(R)$ into a filtered ring. That $F_nR \subset F_nQ_\kappa(R)$ is obvious (we reduced to the torsion free case), on the other hand $E_n = F_nQ_\kappa(R) \cap R$ is the $\kappa$-closure of $F_nR$ in $R$, say $y \in E_n, Iy \subset R$ with $F_\alpha Iy \subset F_{\alpha + n}R$ for all $\alpha \in \mathbb{Z}$. If $y \notin F_nR$, say $y \in \dot{F}_\delta R$ with $\delta > n$, then $F_\alpha I y \subset F_{\alpha + \delta - 1}R$ as $n \leq \delta - 1$. By $\kappa$-separatedness, $y = 0$ a contradiction. Hence $F_nQ_\kappa(R) \cap R = F_nR$ follows. This establishes that $j_\kappa: R \to Q_\kappa(R)$ is a strict filtered map of rings (note that $R \to R/\kappa R$ is already strict so the restriction to the $\kappa$-torsion free case did not harm this).\\
Now for $x \in Q_\kappa(R), y \in Q_\kappa(M)$ with $I_1x \subset R, I_2y \subset M$ both satisfying \eqref{fn}, put $J = (I_2:x)$ and as in the preceding argument it follows that $Q_\kappa(M)$ is a filtered $Q_\kappa(R)$-module. Also the proof that $F_nQ_\kappa(M) \cap M = F_nM$ goes through in the same way as for $R$. Finally observe that $FQ_\kappa(M)$ (resp. $FQ_\kappa(R)$) is separated. Indeed, put $E = \cap F_nQ_\kappa(M)$, then $E \cap M = \cap F_nM = 0$, but for $z \in E, Iz \subset M$ for some $I \in \mc{L}(\kappa)$, thus $Iz \subset E \cap M = 0$ contradicting that $Q_\kappa(M)$ is $\kappa$-torsion free, unless $E = 0$. 
\end{proof}

\begin{example}\exalabel{glider}
Suppose we are in the `almost commutative' case, that is $G(R)$ is a commutative domain (then $Q_\kappa(R)$ has a filtration for every kernel functor $\kappa$ by \corref{almostcom}). Let $M$ be a glider representation, $M \subset \Omega = RM$. We consider on $\Omega$ the filtration $f\Omega$ given by
\begin{equation} \label{filtratienietstandaard}
\begin{array}{l}
F_n\Omega =\sum_{i - j =n} F_iRM_j,\\
F_{-n}\Omega = M_n, \quad n \geq 0.
\end{array}
\end{equation} The associated graded $G_f(\Omega)$ has for the negative part $g(M) = \oplus_{i \geq 0}M_i/M_{i+1}$ and this is itself a fragment with respect to the positive part of the grading filtration of $G(A)$ (see \cite{NVo1}). If $\Omega / \kappa(\Omega)$ is $\kappa$-separated, e.g. when $G_f(\Omega / \kappa(\Omega))$ is $\ov{\kappa}$-torsion free, or when $G_f(\Omega)$ is a faithful $G(R)$-module, then the degree 0 part $F_0Q_\kappa(\Omega)$ becomes a fragment w.r.t the positive part of the filtration $F^+Q_\kappa(R)$ by \exaref{zeropart}. We call $Q_\kappa(M) := F_0Q_\kappa(\Omega)$ the localized fragment of $M$ w.r.t the kernel functor $\kappa$.\\
As a particular case, if $S$ is an Ore set in a standard filtered ring $FR$, we obtain the quotient fragment $Q_S(M)$ from the previous section. In fact the standard assumption is no longer needed in our more general setting.
\end{example}

\begin{example}\exalabel{coordinate} ~\\
$(i)$  Let $V$ be an affine variety with coordinate ring $\Gamma(V)$. Then one considers the ring of differential operators $D(V)$ with the $\Sigma$-filtration. This is a standard filtered ring with $F_0D(V) = \Gamma(V)$ and with graded associated being a commutative domain. In forthcoming work, we will exploit this setting more, investigating for example the link between Ore sets in $D(V)$ and $\Gamma(V)$ (as the previous section uncovered) or the characteristic variety of a glider fragment. For details on rings of differential operators we refer to \cite{Bj}.\\
$(ii)$ Let $V$ and $W$ be varieties embedded in resp. $\mathbb{A}_n(K), \mathbb{A}_m(K)$ such that $\Gamma(V) \subset \Gamma(W)$,
$\Gamma(V) = K[a_1,\ldots, a_n], \Gamma(W) = K[a_1,\ldots,a_m][b_1,\ldots, b_{m-n}]$. We view the $V$-filtration on $\Gamma(W)$ given by
$$F_0\Gamma(W) = \Gamma(V),$$
$$F_1\Gamma(W) = \Gamma(V)b_1 + \cdots + \Gamma(V)b_{n-m} = \Gamma(V)[B^1],$$
where $B^1 = Kb_1 + \cdots + Kb_{n-m}$, $B= \{b_1,\ldots, b_{n-m}\}, \Gamma(W) = \Gamma(V)[B]$. For $n \geq 1$ we define
$$F_n\Gamma(W) = \Gamma(V)[B^n],$$
where $B^n = \sum_{i_1,\ldots, i_n} Kb_{i_1}\ldots b_{i_n}$, $b_{i_j} \in B, j =1, \ldots, n$.\\
$W$-glider representations over $V$ are then by definition $M \subset \Omega$, $\Omega$ a $\Gamma(W)$-module, $M$ an $F\Gamma(W)$-fragment structure induced by the $\Gamma(W)$-module $\Omega$.\\
There are other interesting geometric filtrations on $\Gamma(W)$, e.g. the ring filtration $\Gamma(V) \subset \Gamma(V)[b_1] \subset \ldots \subset \Gamma(V)[b_1, \ldots, b_d] = \Gamma(W)$, where each $\Gamma(V)[b_1,\ldots, b_i]$ corresponds to a $W_i$ with $\Gamma(V) \subset \Gamma(W_i) \subset \Gamma(W)$ and $W_i$ embedded in $\mathbb{A}_{n+i}(K)$ say, $W = W_d \to W_{d-1} \to \ldots \to W_1 \to V$. This commutative theory of glider representations has to be connected to the geometric properties of $V$ and $W$, this is work in progress.
\end{example}

Let $M$ be a filtered $R$-module. A submodule $N$ of $M$ is said to be closed if $\cap_\gamma N + F_\gamma M = N$. If $FM$ is good and if $FR$ is Zariskian then for every $N \subset M$ we have that $N$ is closed (see \cite{LiVo}). In particular if $FM$ is good then $\kappa M$ is closed and therefore $M/\kappa M$ with the induced filtration is separated. Indeed if $y\in M$ maps to $F_\gamma(M/ \kappa M)$ for all $\gamma$, then $y \in \kappa M + F_\gamma M$ for all $\gamma$, hence $y \in \cap_\gamma \kappa M + F_\gamma M = \kappa M$ and $y$ maps to zero. Without the assumption that $FM$ is good we have
\begin{proposition}
If $M$ is $\kappa$-separated then $\kappa M$ is closed and $M/\kappa M$ is separated.
\end{proposition}
\begin{proof}
Suppose $y \in \cap_\gamma \kappa M + F_\gamma M$ and $y \notin \kappa M$. Then $F_\gamma M \neq 0$ even $F_\gamma M \not\subset \kappa M$ for all $\gamma$. Since $FM$ is separated, $y \in \dot{F}_nM$ for some $n \in \mathbb{Z}$. Take $\gamma \leq n-1, y \in \kappa M + F_\gamma M$, say $y = t_\gamma + f_\gamma$ where $t_\gamma \in \kappa M, f_\gamma \neq 0$ in $F_\gamma M$. Choose $I$ such that $It_\gamma = 0$ and look at $F_\tau Iy = F_\tau I f_\gamma \subset F_{\tau + \gamma}M \subset F_{\tau + n - 1}M$. This holds for all $\tau \in \mathbb{Z}$, so since $M$ is $\kappa$-separated, $y \in \kappa M$ follows. Thus $\kappa M$ is closed and $M/\kappa M$ is separated for the induced filtration from $FM$. 
\end{proof}
\begin{corollary}\corlabel{sep}
If $M$ is $\kappa$-separated, then $FQ_\kappa(M)$ is a separated filtration.
\end{corollary}
\begin{proof}
If $\cap_\gamma F_\gamma Q_\kappa(M) = E$, then $E \cap M/ \kappa M = \cap F_\gamma (M / \kappa M)$. By the proposition, $E \cap M/ \kappa M = 0$. Now if $E\neq 0$ then, say $x \neq 0, x \in E, Ix \subset M/\kappa M$ for some $I \in \mc{L}(\kappa)$. Thus $Ix \subset E \cap M/\kappa M = 0$ but then $x \in \kappa(Q_\kappa(M)) = 0$.
\end{proof}
The situation in the above corollary occurs for example when $G(M)$ is $\ov{\kappa}$-torsion free, i.e. when $\ov{\kappa}G(M) = 0$.\\

Suppose now that we have two kernel functors $\kappa$ and $\tau$. We can consider the composition $\tau\kappa$, which in general is only a preradical. In other words, $\tau\kappa$ is not a localization functor, that is the associated torsion class $\mc{T}_{\tau\kappa}$ is not hereditary. The latter is defined by the $R$-modules $M$ such that there exists a submodule $N\subset M$ such that $N$ is $\tau$-torsion and $M/N$ is $\kappa$-torsion, or equivalently $M/\tau(M)$ is $\kappa$-torsion. The associated filter of left ideals is denoted by $\mc{L}(\tau\kappa)$ and consists of left ideals containing a left ideal $\sum_{\alpha}^{'} I_\tau x_\kappa^\alpha$ where $I_\kappa = \sum_\alpha^{'} Rx_\kappa^\alpha \in \mc{L}(\kappa)$ and $I_\tau \in \mc{L}(\tau)$. We denote this by $I_\tau \cdot I_\kappa$ ($FR$ is Zariskian so $R$ is Noetherian and the existence of a finite set of generators for any left ideal $I$ is guaranteed).

\begin{lemma}
The filter $\mc{L}(\tau\kappa)$ satisfies the first three properties from the introduction.
\end{lemma}
\begin{proof}
We only prove the third property. Let $I \cdot J \in \mc{L}(\tau\kappa)$, with $J = \sum_\alpha^{'} Rx_\kappa^\alpha$ and take $r \in R$. Then there exists $y_\kappa^\beta \in R$ with  $(J:r) = \sum_{\beta}^{'}Ry_\kappa^\beta  \in \mc{L}(\kappa)$. Let $ y_\kappa^\beta r = \sum r_{\beta,\alpha} x_\kappa^\alpha \in J$. We find an $I' \in \mc{L}(\tau)$ such that $I' r_{\beta,\alpha} \subset I$ for all $\alpha, \beta$. Then
$$ I' \cdot (J:r)r = \sum_{\beta} I' y_{\kappa}^\beta r = \sum_{\alpha,\beta} I' r_{\beta,\alpha} x_\kappa^\alpha \subset I \cdot J.$$
\end{proof}

However we don't know in general whether $\tau\kappa$ is radical, we still can consider the canonical morphism $j_{\tau\kappa} : M \to Q_\kappa Q_\tau(M)$ for any $R$-module $M$. This is just the composition 
$$ \xymatrix{ M \ar[r]^{j_\tau} & Q_\tau(M) \ar[r]^{j_\kappa^{'}} & Q_\kappa Q_\tau(M),}$$
in which the latter is the $\kappa$-localization morphism of $Q_\tau(M)$. We have 
$$j_{\tau\kappa}(M) = \frac{M/\tau(M)}{M/\tau(M) \cap \kappa Q_\tau(M)} = \frac{M/\tau(M)}{\kappa(M/\tau(M))},$$
whence $j_{\tau\kappa}(M) = 0$ if and only if $M \in \mc{T}_{\tau\kappa}$ and this implies $Q_\kappa Q_\tau (M) \in \mc{T}_{\tau\kappa}$. Hence, for any $R$-module $M$ we have that $Q_\kappa Q_\tau (M)/j_{\tau\kappa}(M) \in \mc{T}_{\tau\kappa}$.

\begin{proposition} \cite[Proposition 3.10]{Vovita}\\
For $M \in R$-mod, the following statements are equivalent
\begin{enumerate}
\item $j_{\tau\kappa}(M) = 0$;
\item $M \in \mc{T}_{\tau\kappa}$;
\item For every $m \in M$ there is an $L \in \mc{L}(\tau\kappa)$ such that $Lm =0$.
\end{enumerate}
\end{proposition}

The above proposition shows that for any $x \neq 0$ in the double localization $Q_\kappa Q_\tau (M)$ there exists an $I \cdot J \in \mc{L}(\tau\kappa)$ such that $I\cdot Jx \subset j_{\tau\kappa}(M)$. One also verifies that the kernel $\Ker(j_{\tau\kappa})$ equals $\tau\kappa(M)$, the elements of $M$ annihilated by an ideal of $\mc{L}(\tau\kappa)$.

\begin{lemma}\lemlabel{tors}
Let $M$ be a filtered $R$-module. Then $ m \in \tau\kappa(M)$ if and only if $ \ov{m} \in \kappa(M/ \tau M)$.
\end{lemma}
\begin{proof}
If $ \ov{m} \in \kappa(M/\tau M)$ with original $m \in M$, then for some $I = \sum_\alpha^{'} Rx_\kappa^\alpha \in \mc{L}(\kappa)$ we have $Im \subset \tau M$. Let $J \in \mc{L}(\tau)$ such that $Jx_\kappa^\alpha m = 0$ for all $\alpha$, then $J \cdot I m = 0$, or $m \in \tau\kappa(M)$. The converse is similar. 
\end{proof}
The previous lemma shows that $\kappa(M/\tau M) = \tau\kappa(M)/\tau(M)$, whence
$$j_{\tau\kappa}(M) = \frac{M /\tau M}{\kappa(M/ \tau M)} = \frac{M/ \tau M}{\tau\kappa{M}/\tau M} \cong M / \tau\kappa{M}.$$
\begin{corollary}
Let $M$ be a filtered $R$-module. Then $M/\tau\kappa M$ is $\tau\kappa$-separated if and only if $(M/\tau M)/\kappa(M/\tau M)$ is $\kappa$-separated.
\end{corollary}

\begin{proposition}\proplabel{filttwo}
Let $\tau, \kappa$ be a kernel functors and $M$ a filtered $R$-module. If $M/\tau\kappa M$ is $\tau\kappa$-separated, then the double localization module $Q_\kappa Q_\tau(M)$ has a filtration $FQ_\kappa Q_\tau(M)$ making the localization morphism $j_{\tau\kappa}(M)$ into a strict filtered morphism. Similar for $R/\tau\kappa R$ being $\tau\kappa$-separated, $j_{\tau\kappa}: R \to Q_\kappa Q_\tau(R)$ is a strict filtered morphism and a filtered ring morphism. \end{proposition}
\begin{proof}
One follows the proof of \theref{locfilt} without reducing to the torsion free case ($\tau\kappa$ no longer radical!) and one checks that only the first three properties for filters of ideals are used. Explicitly
$$F_\gamma Q_\kappa Q_\tau (M) = \{ q \in Q_\kappa Q_\tau(M), \exists I \cdot J \in \mc{L}(\tau\kappa) {\rm~such~that~} F_n(I \cdot J)q \subset F_{n+q}j_{\tau\kappa}(M) {\rm~for~all~} n \in \mathbb{Z} \}.$$
\end{proof}

If the filtered module $M/\tau M$ is $\tau$-separated, we obtain a filtered morphism $FQ_\tau(M)$. If $Q_\tau(M)/\kappa(Q_\tau(M))$ is $\kappa$-separated, we get a filtration $FQ_\kappa Q_\tau(M)$. If however, $M/\tau\kappa(M)$ is also $\tau\kappa$-separated, then we obtain a second filtration $F'Q_\kappa Q_\tau(M)$. We clearly have an inclusion $F' \subset F$ of filtered modules. In case $\kappa$ is coming from a $\wt{\kappa}$ (see second approach below), then we obtain equality, for in this case we have for $I = \sum_\alpha^{'}Rx_\kappa^\alpha \in \mc{L}(\kappa)$ and $J \in \mc{L}(\tau)$ that $F_\gamma(J \cdot I) = \sum_{\alpha}F_{\gamma - \deg(x_\kappa^\alpha)}J x_\kappa^\alpha$ for any $\gamma$, since $(\wt{J}\wt{I})_\gamma = \sum_{\sigma + \tau = \gamma}\wt{J}_\sigma \wt{I}_\tau$. We do not know whether we have equality in general. If one works with Ore sets in the `almost commutative' case everything holds on the other hand
\begin{lemma}
Let $S, T$ be Ore sets in $R$ with $G(R)$ being a commutative domain. Then under the above assumptions, $F_nQ_S Q_T(M) \subset  F_n'Q_S Q_T(M)$ for all $n \in \mathbb{Z}$.
\end{lemma}
\begin{proof}
For $s \in S, t \in T$ we have that $Rt \cdot Rs = Rts$ and $F_\gamma(Rts) = F_{\gamma - \deg(st)}R st$, so the conclusion easily follows.
\end{proof}

For two Ore sets $S, T$ we have the following diagram for an $R$-filtered module $M$.
\begin{equation} \label{arrows}
\xymatrix{
Q_T(M) \ar[rr] \ar[rrd] && Q_T(Q_S(M)) \\
Q_S(M) \ar[rr] \ar[rru] && Q_S(Q_T(M))}.
\end{equation}
Suppose that $M$ is such that the corresponding quotients with induced filtrations are $S$-, $T$-, $ST$-, resp. $TS$-separated, that $Q_T(M)/S Q_T(M)$ is $S$-separated and that $Q_S(M)/T Q_S(M)$ is $T$-separated. Both diagonal morphisms are the canonical localization morphisms whence filtered. The upper horizontal arrow is the localized morphism $Q_T(j_S)$, which is also filtered by
\begin{lemma}\lemlabel{filteredmorphism}
Let $f:M \to N$ be a filtered morphism between filtered $FR$-modules $M,N$ and $\kappa$ an idempotent kernel functor. If both $M$ and $N$ are $\kappa$-separated, then $Q_\kappa(f): Q_\kappa(M) \to Q_\kappa(N)$ is a filtered morphism with respect to the localization filtrations.
\end{lemma}
\begin{proof}
This follows from the fact that $Q_\kappa(f)$ is (left) $R$-linear and from the definition of the localization filtration.
\end{proof}

\begin{proposition}\proplabel{filteredmor}
If $\kappa \leq \tau$ and $M$ is $\kappa$-separated and $\tau$-separated then there is a canonical filtered morphism $\rho^\kappa_\tau:Q_\kappa(M) \to Q_\tau(M)$ yielding a commutative diagram
$$\xymatrix{
Q_\kappa(M) \ar[r]^{\rho^\kappa_\tau} & Q_\tau(M)\\
M/\kappa(M) \ar[r]^{\pi^\kappa_\tau} \ar@{^{(}->}[u]_{j_\kappa} & M/ \tau(M)  \ar@{^{(}->}[u]_{j_\tau}}$$
where the vertical morphisms are strict filtered injections and $\pi^\kappa_\tau$ is the canonical filtered epimorphism.
\end{proposition}
\begin{proof}
The statements about $j_\kappa, j_\tau$ and $\pi^\kappa_\tau$ are obvious. We know that $\pi^\kappa_\tau$ extends to an $R$-linear $\rho^\kappa_\tau$ by general localization theory. Now let $y \in F_nQ_\kappa(M)$, that is, there is an $I \in \mc{L}(\kappa)$ such that for all $\gamma \in \mathbb{Z}, F_\gamma I y \subset F_{\gamma +n} M/ \kappa M$. Applying $\rho^\kappa_\tau$ yields $F_\gamma I \rho^\kappa_\tau(y) = \rho^\kappa_\tau(F_\gamma I y) = \pi^\kappa_\tau(F_\gamma I y) \subset F_{\gamma + n}(M / \tau M)$ since $\pi^\kappa_\tau$ is filtered. Now $F_\gamma I \rho^\kappa_\tau(y) \subset F_{\gamma + n}(M /\tau M)$ for all $\gamma$ means that $\rho^\kappa_\tau(y) \in F_nQ_\tau(M)$ and thus $\rho^\kappa_\tau$ is a filtered morphism.
\end{proof}
With an eye toward the noncommutative site, we add an additional kernel functor $\sigma$. Since $\frac{M}{\tau\kappa(M)} = \frac{M/\tau(M)}{\kappa(M/\tau(M))}$, we have that $\mc{T}_{(\tau\kappa)\sigma} = \mc{T}_{\tau(\kappa\sigma)}$. We denote this torsion class by $\mc{T}_{\tau\kappa\sigma}$ and the associated filter $\mc{L}(\tau\kappa\sigma)$ consists of left ideals containing some $I_\tau \cdot I_\kappa \cdot I_\sigma = \sum_{\alpha,\beta}^{'} I_\tau y_\kappa^\beta x_\sigma^\alpha$ with $I_\tau \in \mc{L}(\tau), I_\kappa = \sum_{\beta}^{'} Ry_\kappa^\beta \in \mc{L}(\kappa), I_\sigma = \sum_{\alpha}^{'} Rx_\sigma^{\alpha} \in \mc{L}(\sigma)$. This filter also satisfies the first three properties from the introduction and the canonical localization morphism
$$j_{\tau\kappa\sigma}: M \mapright{j_\tau} Q_\tau(M) \mapright{j_\kappa^{'}} Q_\kappa Q_\tau(M) \mapright{j_\sigma^{'}} Q_\sigma Q_\kappa Q_\tau(M),$$
has image
$$\frac{M/\tau(M)}{\kappa(M/\tau(M))} / \sigma\left(\frac{M/\tau(M)}{\kappa(M/\tau(M))}\right),$$
which equals $M/ \tau\kappa\sigma(M)$ by \lemref{tors}. Hence, we again have that \newline $Q_\sigma Q_\kappa Q_\tau(M)/ j_{\tau\kappa\sigma}(M) \in \mc{T}_{\tau\kappa\sigma}$ and if $M/\tau\kappa\sigma(M)$ is $\tau\kappa\sigma$-separated, the localization module $Q_\sigma Q_\kappa Q_\tau(M)$ has a filtration $FQ_\sigma Q_\kappa Q_\tau(M)$ such that $j_{\tau\kappa\sigma}$ is strict filtered Observe that in \propref{filteredmor} it actually follows that the $\rho^\kappa_\tau$ is strict. Indeed, $M/ \tau(M) = \frac{M / \kappa(M)}{\tau(M)/\kappa(M)}$ with induced filtration of $M/\kappa(M)$ on $\tau(M)/\kappa(M)$ shows that $\pi^\kappa_\rho$ is strict, which shows strictness of $\rho^\kappa_\tau$.The previous observations imply that the $\rho^{W'}_W: Q_{W'}(M) \to Q_W(M)$ for words in kernel functors $W' \hookmapright{} W$ are also strict filtered maps. Moreover, one checks that the results following \propref{filttwo} remain valid (up to restriction to Ore sets) and that we can extend diagram \eqref{arrows}. In fact, one can consider words of finite length of kernel functors. In this paper we will restrict to Ore sets when dealing with the noncommutative site, but we chose the more general approach using kernel functors as the results most likely will be extendable to the so called noncommutative affine site.\\

Now, we start with a graded kernel functor $\widetilde{\kappa}$ on $\widetilde{R}$-mod and we define 
$$\mc{L}(G\kappa) = \{ L \subset G(R), L \supset \pi(I) {\rm~for~some~} I \in \mc{L}(\widetilde{\kappa})\}.$$
where $\pi: \widetilde{R} \to G(R), \widetilde{r} \mapsto \widetilde{r} \mod \widetilde{R} X$. 
\begin{proposition}\proplabel{Gk}
$G\kappa$ is a kernel functor.
\end{proposition}
\begin{proof}
a) If $L' \supset L$ with $L \in \mc{L}(G\kappa)$, then $L' \in \mc{L}(G\kappa)$.\\
b) If $L_1,L_2 \in \mc{L}(G\kappa)$ then $L_1 \cap L_2 \supset \pi(I_1 \cap I_2), I_1 \cap I_2 \in \mc{L}(\widetilde{\kappa})$.\\
c) Look at $\ov{r} \in G(R)$ and $\widetilde{r} \in \widetilde{R}$ such that  $\pi(\widetilde{r}) = \ov{r}$. Take $L \in \mc{L}(G\kappa)$ and $(L : \ov{r}) \subset G(R)$. Pick $I \in \mc{L}(\widetilde{\kappa})$ such that $L \supset \pi(I)$ then $(I : \widetilde{r}) \in \mc{L}(\widetilde{\kappa})$ since the latter is a Gabriel filter. From $(I : \widetilde{r})\widetilde{r} \subset I$ it follows that $\pi(I : \widetilde{r})\ov{r} \subset \pi(I) \subset L$ with $\pi(I: \widetilde{r}) \in \mc{L}(G\kappa)$, whence $(L:\ov{r}) \in \mc{L}(G\kappa)$.\\
d) Consider $L \in \mc{L}^g(G\kappa)$ and a graded $J \subset L$ such that for all $l \in L$ there is an $I_l \in \mc{L}(G\kappa)$ such that $I_l l \subset J$. Since $L \supset \pi(I)$ for some $I \in \mc{L}(\widetilde{\kappa})$ we may put $K = \pi^{-1}(J) \cap I$, thus $\pi(K) \subset J \cap \pi(I)$. If $\widetilde{i} \in I$, say $l = \pi(\widetilde{i}) \in L$, then there is an $I_l \in \mc{L}(G\kappa)$ such that $I_l l \subset J$. Thus $J_l \widetilde{i} \subset \pi^{-1}(J) \cap I = K$, with $J_l \in \mc{L}(\widetilde{\kappa})$ such that $\pi(J_l) \subset I_l$. This holds for every $\widetilde{i} \in I$, thus $K \in \mc{L}(\widetilde{\kappa})$ and $\pi(K) \in \mc{L}(G\kappa)$. This implies that $J \cap \pi(I) \in \mc{L}(G\kappa)$ and $J \in \mc{L}(G\kappa)$.\\
Since $\mc{L}(\widetilde{\kappa})$ is graded, so is $\mc{L}(G\kappa)$. Therefore, if $ L \in \mc{L}(G\kappa)$ and $J \subset L$ are not graded, then $L \supset \pi(I)$ still holds with $\pi(I)$ a graded ideal and $J \cap \pi(I) \subset \pi(I)$. Let $J^g \subset J$ be the largest graded $G(R)$-submodule of $J$.  Since $(J\cap \pi(I)) / (J^g \cap \pi(I))$ and $\pi(I)/(J \cap \pi(I))$ are $G\kappa$-torsion, so is $\pi(I)/J^g \cap \pi(I)$. By the first part of this proof $J^g \cap \pi(I)$ and thus $J^g \in \mc{L}(G\kappa)$. It follows that $J \in \mc{L}(G\kappa)$.
\end{proof}

We call $G\kappa$ the graded kernel functor on $G(R)$-mod induced by $\widetilde{\kappa}$ on $\widetilde{R}$-mod. We also write $G\kappa = \ov{\kappa}$. Sometimes $\ov{\kappa}$ is trivial, e.g. when $\wt{R}X \in \mc{L}(\wt{\kappa})$ then $0 \in \mc{L}(\ov{\kappa})$ and $\ov{\kappa}$ is indeed trivial. If $\mc{L}^g(\wt{\kappa})$ does not contain $\wt{R}X$ then $\ov{\kappa}$ is non-trivial. Observe moreover that the $\ov{\kappa}$ defined here from $\wt{\kappa}$ is not the same as the $\ov{\kappa}^1$ defined from $\kappa$ as in the first approach. Recall that $\ov{\kappa}^1$ has a filter $\mc{L}(\ov{\kappa}^1) = \{ L \subset G(R), L \supset \pi(\wt{I}) {\rm ~for~some~} I \in \mc{L}(\kappa)\}$, whilee $\mc{L}(\ov{\kappa}) =  \{ L \subset G(R), L \supset \pi(I) {\rm ~for~some~} I \in \mc{L}(\wt{\kappa})\}$. Both filters coincide exactly when $\mc{L}(\ov{\kappa})$ has a filterbasis consisting of $\wt{I}$ with $I \in \mc{L}(\kappa)$, which is the filter in $R$-mod defined by
$$\mc{L}(\kappa) = \{ J \subset R, \widetilde{J} \in \mc{L}(\widetilde{\kappa})\},$$
The $\widetilde{J}$ are always calculated with respect to the $FJ$ induced by $FR$ on $J$. Graded kernel functors $\wt{\kappa}$ on $\wt{R}$-mod satisfying the above property are called pseudo-affine. For now, we forget about this property. Our definition of $\mc{L}(\kappa)$ is satisfactory since
\begin{proposition}
$\mc{L}(\kappa)$ is a Gabriel filter, that is, the associated preradical $\kappa$ is a kernel functor.
\end{proposition}
\begin{proof}
a) Consider $H \supset J$ with $J \in \mc{L}(\kappa)$, then $\widetilde{H} \supset \widetilde{J}$ yields $\widetilde{H} \in \mc{L}(\widetilde{\kappa})$.\\
b) Take $J_1,J_2 \in \mc{L}(\kappa)$ then $(J_1 \cap J_2)^\sim = \widetilde{J_1} \cap \widetilde{J_2}$. Indeed $(J_1 \cap J_2)^\sim \subset\widetilde{J_1} \cap \widetilde{J_2}$ is obvious. If $\widetilde{j} \in \widetilde{J_1} \cap \widetilde{J_2}$ then $\widetilde{j} \in \widetilde{J_1}$ implies that $\widetilde{j} = j_1X^n$ where $j_1 \in J_1$, $n = \deg(j_1)$. $\widetilde{j} \in \widetilde{J_2}$ yields $\widetilde{j} = j_2X^n$ where $n = \deg\widetilde{j}$ and $\deg(j_2) = n$. From $(j_1 - j_2)X^n = 0$ and $\widetilde{R}$ being $X$-torsion free it follows that $j_1 = j_2$, hence $\widetilde{j} = jX^n$ with $j \in J_1 \cap J_2$, or $\widetilde{J_1} \cap \widetilde{J_2} \subset (J_1 \cap J_2)^\sim$.\\
c) Take $r \in R, J \in \mc{L}(\kappa)$, then $\widetilde{r} \in \widetilde{R}$ and $\widetilde{J} \in \mc{L}(\widetilde{\kappa})$. $(\widetilde{J}: \widetilde{r}) \in \mc{L}(\widetilde{\kappa})$. Let $y \in (J:r)$, i.e. $\widetilde{y} \in (J:r)^\sim$. If $\deg(yr) = \gamma$ with $\gamma \leq \deg(y) + \deg(r) = m + n$, then $\widetilde{yr} = yrX^\gamma, \widetilde{y}\widetilde{r} = \widetilde{yr}X^{ m + n - \gamma}$. Thus $(yr)^\sim \subset \widetilde{J}$ yields $\widetilde{y}\widetilde{r} \in X^{n+m - \gamma}\widetilde{J} \subset \widetilde{J}$, i.e. $\widetilde{y} \in (\widetilde{J}:\widetilde{r})$. Conversely, from $\wt{z}\wt{r} \in \wt{J}$ we have $\wt{z} = zX^q$ for some $z \in R$, then $zrX^{q+n} \in \wt{J}$. $zr \in \dot{F}_\gamma R$ for some $\gamma \leq q +n$, whence $zrX^{q+n} = zrX^\gamma X^{q + n - \gamma} \in \wt{J}$. Since $J \subset R$ has the induced filtration $\wt{R}/ \wt{J}$ is $X$-torsion free. Therefore $zrX^\gamma = \wt{zr} \in \wt{J}$, or $zr \in J$ and $z \in (J:r)$. So we have established $(J:r)^\sim = (\wt{J} : \wt{r})$ and $\wt{J} \in \mc{L}(\wt{\kappa})$ then yields $(J:r) \in \mc{L}(\kappa)$.\\
d) Take $J \in \mc{L}(\kappa), H \subset J$ such that $\forall j \in J, \exists I_j \in \mc{L}(\kappa)$ such that $I_j j \subset H$. So we have $\wt{J} \in \mc{L}(\wt{\kappa}), \wt{H} \subset \wt{J}$ ($FH$ induced by $FR$ like $FJ$). Let us check $\wt{I_j}\wt{j} \subset \wt{H}$ for all homogeneous $\wt{j} \in \wt{J}$; then idempotency of $\wt{\kappa}$ (and the fact that $\wt{\kappa}$ is graded) yields $\wt{H} \in \mc{L}(\wt{\kappa})$ or $H \in \mc{L}(\kappa)$. For $\wt{i_j} \in \wt{I_j}$, $\wt{i_j}\wt{j} = (i_j j)^\sim X^{n_i + n_j - n'}, n_i = \deg\wt{i_j}, n_j = \deg\wt{j}$ and $n' = \deg i_j j$. Thus $\wt{i_j}\wt{j} \in X^{n_i + n_j - n'} \wt{H} \subset \wt{H}, n_i + n_j - n' \geq 0$. This holds for all homogeneous $\wt{i_j} \in \wt{I_j}$, hence $\wt{I_j}\wt{j} \subset \wt{H}$. 
\end{proof}

\begin{lemma}\lemlabel{tilde}
With assumptions and notations as before: For a separated filtered $FR$-module $M$, $\wt{\kappa}(\wt{M}) = \kappa(M)^\sim$ and $\wt{M}/\wt{\kappa}(\wt{M}) = (M/\kappa M)^\sim$.
\end{lemma}
\begin{proof}
Obviously $\wt{\kappa}(\wt{M})$ is $X$-torsion free, thus $\wt{\kappa}(\wt{M}) = \wt{N}$ for some filtered $R$-submodule $N$ in $M$. If $\wt{m} \in \wt{M}$ is $X$-torsion modulo $\wt{\kappa}(\wt{M})$, i.e. $X^d\wt{m} \in \wt{\kappa}(\wt{M})$ then $\wt{I}X^d\wt{m} = 0$ for some $\wt{I} \in \mc{L}(\wt{\kappa})$, hence $\wt{I}\wt{m} = 0$ or $\wt{m} \in \wt{\kappa}(\wt{M})$. Consequently $\wt{\kappa}(\wt{M})$ is $X$-closed in $\wt{M}$, which means that $N$ is a strict filtered submodule of $M$, i.e. $FN$ is induced by $FM$. If $n \in N$ then $\wt{n} \in \wt{N}$ is $\wt{\kappa}$-torsion so $L\wt{n} = 0$ for some $L \in \mc{L}(\wt{\kappa})$. We have $\cl_X(L) \in \mc{L}(\wt{\kappa})$ and since $FR$ is Zariskian, $\cl_X(L)$ is finitely generated, hence there is an $n \in \mathbb{N}$ such that $X^n\cl_X(L) \subset L$, as $X$ is central in $\wt{R}$. Then $X^n\cl_X(L)\wt{n} = 0$ yields $\cl_X(L)\wt{n} =0$ since $\wt{N}$ is $X$-torsion free. Now $\cl_X(L)$ is $X$-closed in $\wt{R}$ hence $\cl_X(L) =  \wt{I}$ for some $I$ and $I \in \mc{L}(\kappa)$ since $\cl_X(L) \in \mc{L}(\wt{\kappa})$. Now from $\wt{i}\wt{n} = \wt{in}X^\delta$ where $\delta = \deg_F i + \deg_F n - \deg_F in$, $\deg_F$ being defined since $FM$ is separated, we obtain  $in = 0$ and this holds for all $i \in I$ with $I \in \mc{L}(\kappa)$, thus $n \in \kappa(M)$ and $\wt{n} \in (\kappa(M))^\sim$. Consequently $(\kappa(M))^\sim \supset \wt{\kappa}(\wt{M})$. Conversely if $\wt{m}$ is homogeneous in $(\kappa(M))^\sim$ then $\wt{m} = mX^d$ for some $m \in \dot{F}_d(\kappa(M))$, thus $Lm = 0$ for some $L \in \mc{L}(\kappa)$. Then from $(Lm)^\sim \supset \wt{L}\wt{m}$ it follows that $\wt{m} \in \wt{\kappa}(\wt{M})$, therefore $\wt{\kappa}(\wt{M}) = (\kappa(M))^\sim$.\\
We checked that $\wt{M}/\wt{\kappa}(\wt{M})$ is $X$-torsion free above. Moreover $\wt{M}/\wt{\kappa}(\wt{M}) = \wt{M} / \kappa(M)^\sim$. We have an exact sequence of strict filtered morphisms:
$$ 0 \to \kappa(M) \to M \to M/\kappa(M) \to 0$$
Therefore, applying $\sim$ to this sequence. we get an exact sequence in $\wt{R}$-gr: 
$$ 0 \to \kappa(M)^\sim \to \wt{M} \to (M/\kappa(M))^\sim \to 0$$
Thus $\wt{M} / \wt{\kappa}(\wt{M}) = \wt{M} / (\kappa(M))^\sim = (M / \kappa(M))^\sim$.
\end{proof} 

We call a graded kernel functor $\wt{\kappa}$ on $\wt{R}$-mod or a kernel functor $\wt{\kappa}$ on $\wt{R}$-gr  ``affine" if $\wt{R}X \in \mc{L}(\wt{\kappa})$. In this case we write $\wt{\kappa} \in$ Afftors$(\wt{R})$. 
\begin{observation}~
\begin{enumerate}
\item If $\wt{\kappa}$ is affine, then $0 \in \mc{L}(\ov{\kappa})$, and localization on the associated graded level is trivial.
\item $\wt{\kappa}$ is affine if and only if: $I \in \mc{L}(\wt{\kappa})$ if and only if $X^nI \in \mc{L}(\wt{\kappa})$ for all $n\in \mathbb{N}$. \\
\emph{Proof.} If $\wt{R}X \in  \mc{L}(\wt{\kappa})$ and $I \in  \mc{L}(\wt{\kappa})$ then since $FR$ is Zariskian, $I$ is finitely generated so $X^nI = \wt{R}X^n I \in  \mc{L}(\wt{\kappa})$ since $\wt{R}X \in  \mc{L}(\wt{\kappa})$. Conversely $X^nI \in  \mc{L}(\wt{\kappa})$ yields for $I = \wt{R} \in  \mc{L}(\wt{\kappa})$ that $\wt{R}X^n$ and $\wt{R}X$ are in $ \mc{L}(\wt{\kappa})$.
\item $\wt{\kappa}$ is affine if $I \in  \mc{L}(\wt{\kappa})$ if and only if $\cl_X(I) \in  \mc{L}(\wt{\kappa})$.\\
\emph{Proof.} Suppose $\cl_X(I) \in  \mc{L}(\wt{\kappa})$, since $\cl_X(I)$ is finitely generated and $X$ is central in $\wt{R}$, there is an $X^n, n \in \mathbb{N}$ such that $X^n\cl_X(I) \subset I$, thus $\wt{R}X^n\cl_X(I) \subset I$ and hence $I \in  \mc{L}(\wt{\kappa})$.
\item  If $\wt{\kappa}$ is affine then $\mc{L}(\wt{\kappa})$ has a filter basis consisting of $X^n\wt{I}$ with $n \in \mathbb{N}$ and $\wt{I}$ such that $\wt{R}/\wt{I}$ is $X$-torsion free or $\cl_X(\wt{I}) = \wt{I}$, in other words $\wt{I}$ is coming from $I \in \mc{L}(\kappa)$.\\
\emph{Proof.} From observation 3. $I \in \mc{L}(\wt{\kappa})$ iff $\cl_X(I) \in  \mc{L}(\wt{\kappa})$ hence $\cl_X(I) = \wt{J}$ for some $J \in \mc{L}(\kappa)$ and $X^n\wt{J} \subset I$ for some $n \in \mathbb{N}$ as in observation 3., thus $I$ contains some $X^n\wt{J}$ as claimed.
\end{enumerate}
\end{observation}

The complement of affine (graded) localizations of $\wt{R}$-mod consists of the $\wt{\kappa}$ such that $\wt{R}X \notin  \mc{L}(\wt{\kappa})$. Then it is obvious that $0 \notin \mc{L}(\ov{\kappa})$. The noncommutative geometry of $\wt{R}$-mod has an ``open set" of affine torsion theories; this can be identified with $\wt{R}_X$-grtors where $\wt{R}_X \cong R[X,X^{-1}]$ via the torsion theories such that $\wt{\kappa} \geq \wt{\kappa}_X$\footnote{For the lattice structure on $R$-tors we refer to Section 4.} where $\wt{\kappa}_X$ is the kernel functor associated to $\mc{L}(\wt{\kappa}_X) = \{ \wt{R}X^n, n \in \mathbb{N}\}$ and it is known that $\wt{R}_X$-grtors $\cong \gen(\wt{\kappa}_X)$ because $\wt{\kappa}_X$ is central and perfect. Since $R[X,X^{-1}]$ is strongly graded, we have an equivalence of categories $\wt{R}_X$-gr $\cong R$-mod, justifying the statement in the introduction. For non-affine $\wt{\kappa}$ we do not necessarily have $\kappa$ being trivial but in any case $\wt{\kappa}$ is not obtained as an obvious $\sim$-construction from $\kappa$. Indeed, starting from a kernel functor $\kappa$ one could attempt to define $\mc{L}(\wt{\kappa})$ as the filter of left ideals containing some $\wt{I}$ for some $I \in \mc{L}(\kappa)$. However, in proving idempotency one needs the assumption that $X^nJ \in \mc{L}(\wt{\kappa})$ if $J \in \mc{L}(\wt{\kappa})$, but this is just saying that $\wt{\kappa}$ is affine. In the set of non-affine $\wt{\kappa}$ there is nonetheless a subset which relates well to both $\kappa$ and $\ov{\kappa}$, these are the \emph{pseudo-affine kernel functors}.

\begin{definition}
$\wt{\kappa}$ is said to be pseudo-affine if $\mc{L}(\wt{\kappa})$ has a filter basis consisting of $\wt{I}$ with $I \in \mc{L}(\kappa)$.
\end{definition}

In situations where both $\kappa$ and $\ov{\kappa}$ are used or linked we have to restrict to pseudo-affine kernel functors. The most important example is given by Zariskian filtered rings with a domain for the associated graded ring and localizations deriving from Ore sets (see below).\\

Starting from a $\kappa$ on $R$-mod we may define a $\mc{L}(\wt{\kappa})$ on $\wt{R}$-gr by putting $\mc{L}(\wt{\kappa}) = \{ J \subset_g \wt{R}, J \supset \wt{I} {\rm~for~some~} I \in \mc{L}(\kappa)\}$. One easily verifies that the $\wt{\kappa}$ defines a left exact preradical, but it is a priori not idempotent.\\
Even if $\kappa$ is defined from a kernel functor $\wt{\kappa}$ then the $\mc{L}(\wt{\kappa})$ is not obtained from $\mc{L}(\kappa)$ by taking the $\wt{I}, I \in \mc{L}(\kappa)$ for a filterbasis! Of course for a pseudo-affine $\wt{\kappa}$ the $\kappa$ defined as before does yield the $\wt{\kappa}$ in the procedure sketched above. Another way trying to solve this is to define $\mc{L}(\kappa)$ as the filter generated by the filtered left ideals $L$ of $R$ such that $\wt{L} \in \mc{L}(\kappa)$, i.e. not viewing only filtered left ideals $I$ such that $0 \to I \to R$ is a strict filtered exact chain! In fact, in this way we consider the category $R$-filt which has been neglected (see Introduction). This may be possible but the whole torsion theory has to be adapted, so we do not go deeper into this.\\
For pseudo-affine kernel functors we obviously have 
$$ \mc{L}(\ov{\kappa}) = \{L \subset G(R), L \supset \pi(\wt{I}) {\rm~for~some~} I \in \mc{L}(\kappa)\}.$$
As far as sheaves over the lattice $R$-tors will go, we will work only with localizations and quotient filtrations defined in the first approach, say for $R$ with $G(R)$ a domain. As soon as one desires good relations on the sheaf level between the $R$-, the $\wt{R}$- and the $G(R)$-levels, one restricts attention to pseudo-affine graded kernel functors $\wt{\kappa}$ on $\wt{R}$-mod and the associated $\kappa$ on $R$-mod and $\ov{\kappa}$ on $G(R)$-mod. In the next section we will discuss for example the noncommutative site in the `almost commutative' case, where words of Ore sets are considered. To this extent, we have

\begin{lemma}
Let $\wt{S}$ be an Ore set in $\wt{R}$ such that $\wt{S} \cap \wt{R}X = \emptyset$. Define $S \subset R, S = \{ s, \wt{s} \in \wt{S}\}$. Then $S$ is an Ore set in $R$.
\end{lemma}
\begin{proof}
Take $s,t \in S$, then $\wt{s}\wt{t} = (st)^\sim X^{n(s) + n(t) - n(st)}$ where $n(s) = \def \wt{s}, n(t) = \deg \wt{t}$ and $n(st) = \deg (st)^\sim$. If $n(st) < n(s) + n(t)$ then $\wt{s}\wt{t} \in \wt{S}$ and $\wt{s}\wt{t} \in \wt{R}X$, contradiction. Hence $n(st) = n(s) + n(t)$ and $st \in S$. Also observe that $S$ is an Ore set. Indeed take $s \in S, r \in R$ and look at $\wt{s} \in \wt{S}, \wt{r} \in \wt{R}$ then $\wt{s'} \wt{r} = \wt{r'} \wt{s}$ because $\wt{S}$ is an Ore set. Put $\deg_F s'r = \gamma, \deg_F r's = \tau$ then 
$$\begin{array}{c}
(s' r)^\sim =  s'rX^\gamma, \wt{s'}\wt{r} = (s'rX^\gamma)X^{\deg(s') + \deg(r) - \gamma}\\
(r's)^\sim = r's X^\tau, \wt{r'}\wt{s} = (r'sX^{\tau})X^{\deg(r') + \deg(s) - \tau},
\end{array}$$
where $\deg(s') + \deg(r) = \deg(r') + \deg(s) = \epsilon$ because $\wt{s'} \wt{r} = \wt{r'} \wt{s}$. We conclude that $s'r = r's$. Furthermore if $rs = 0$ for some $r \in R, s\in S$, then $\wt{r}\wt{s} = (rs)^\sim X^\epsilon$ where $\epsilon = \deg(\wt{r}) + \deg(\wt{s}) - \deg(rs)^\sim, \epsilon \geq 0$. Thus $\wt{r}\wt{s} = 0$ and since $\wt{S}$ is an Ore set in $\wt{R}$, there is an $\wt{s'}$ in $\wt{S}$ such that $\wt{s'}\wt{r} = 0$, hence $s'r=0$.
\end{proof}
To an Ore set as in the lemma, we associate $\ov{S} = \pi(\wt{S}),$ where $\pi : \wt{R} \to G(R)$ is the canonical epimorphism. Since $\wt{S} \cap \wt{R}X = \emptyset, 0 \notin \ov{S}$ and obviously $\ov{S}$ is multiplicatively closed. 
\begin{lemma}
In the situation of the previous lemma, $\ov{S}$ satisfies the second Ore condition. In case $G(R)$ is commutative or a domain then $\ov{S}$ is an Ore set.
\end{lemma}
\begin{proof}
Given $\ov{r} \in G(R), \ov{s} \in \ov{S}$ and $\wt{r}$, resp. $\wt{s}$ representatives in $\wt{R}$, resp. $\wt{S}$. Then $\exists \wt{s'} \in \wt{S}, \wt{r'} \in \wt{R}$ such that $\wt{s'}\wt{r} = \wt{r'} \wt{s}$, consequently $\pi(\wt{s'})\pi(\wt{r}) = \pi(\wt{r'})\pi( \wt{s})$ or $\ov{s'} \ov{r} = \ov{r'} \ov{s}$ for some $\ov{s'} \in \ov{S}, \ov{r'} \in G(R)$. If $G(R)$ is commutative and $\ov{r}\ov{s} = 0$ then $\ov{s}\ov{r} = 0$ and the first Ore condition holds with $\ov{s'} = \ov{s}$. If $G(R)$ is a domain then $\ov{r}\ov{s} = 0$ with $\ov{s} \neq 0$ yields $\ov{r} = 0$ and of course $\ov{s}\ov{r} = 0$.
\end{proof}
If $G(R)$ is $\ov{S}$-torsion free on the right then in the situation of the lemma's, $\ov{S}$ is an Ore set. The foregoing results also work for right Ore sets and for left and right Ore sets. So summarizing all of this yields:
\begin{proposition}\proplabel{tik}
\item If $G(R)$ is $\pi(\wt{S})$-torsionfree left and right and $\wt{S}$ is as in the lemma's, then $S$ and $\ov{S}$ are Ore sets (left if $\wt{S}$ is left, left and right if $\wt{S}$ is left and right). Moreover, $\wt{\kappa}_{\wt{S}}$ defines $\kappa_S$ and $\ov{\kappa}_{\ov{S}}$ because $\wt{\kappa}_{\wt{S}}$ is pseudo-affine.
\end{proposition}

In the `almost commutative' case, we have for an Ore set $S$ in $R$ that $\wt{S} = \{ \wt{s}, s \in S\}$ is an Ore set in $\wt{R}$ not intersecting $\wt{R}X$. For a word in Ore sets, say $S_1S_2\ldots S_n$, we have that $(Rs_1\ldots s_n)^\sim = R\wt{s_1}\ldots\wt{s_n}$ ($G(R)$ domain!), so the preradical $\kappa_{S_1}\kappa_{S_2}\ldots \kappa_{S_n}$ with associated filter $\mc{L}(\kappa_{S_1}\kappa_{S_2}\ldots \kappa_{S_n})$ is indeed pseudo-affine.\\

Let's return to our aim in defining a quotient filtration. The following two lemmas and corollary correspond to \lemref{1}, \lemref{2} and \lemref{tf} from the results starting from a kernel functor $\kappa$ on $R$-mod. Since our approach is different, we do include the proofs in this setting, although they are similar. 

\begin{lemma}
$G(\kappa(M)) \subset \ov{\kappa} G(M)$ and there is an epimorphism \newline $G(M/\kappa(M)) \twoheadrightarrow  G(M)/ \ov{\kappa}G(M)$.
\end{lemma}
\begin{proof}
Note that $G(\kappa(M)) = \pi_{\wt{M}}(\kappa(M)^\sim) = \pi_{\wt{M}}(\wt{\kappa}(\wt{M})$, where $\pi_{\wt{M}}: \wt{M} \to G(M) = \wt{M}/X\wt{M}$ is the canonical epimorphism. Hence $G(\kappa(M))$ is $\ov{\kappa}$-torsion since $\wt{I}\wt{m} = 0$ for $\wt{m} \in \kappa(M)^\sim$ and $\wt{I} \in \mc{L}(\wt{\kappa})$ yields $\wt{I} \pi_{\wt{M}}(\wt{m}) = \pi(\wt{I}) \pi_{\wt{M}}(\wt{m}) = 0$ with $\pi(\wt{I}) \in \mc{L}(\ov{\kappa})$. Thus $G(\kappa(M)) \subset \ov{\kappa}(G(M))$. 
\end{proof}
\begin{lemma}\lemlabel{sep2}
If $M$ is $\kappa$-separated and $\kappa$ is pseudo-affine or else if $M/\kappa(M)$ is $\kappa$-separated then $G(\kappa(M)) = \ov{\kappa}(G(M))$ and \\
$G(M / \kappa(M)) = G(M) / \ov{\kappa}G(M)$.
\end{lemma}
\begin{proof}
Take $\ov{z} \in \ov{\kappa}(G(M))_n$, we want to establish that $\ov{z} \in G(\kappa(M))$. The relation between $\kappa, \ov{\kappa}$ and $\wt{\kappa}$ entail that for $\ov{J} \in \mc{L}(\ov{\kappa})$ there is a $J \in \mc{L}(\kappa)$ such that $\ov{J} \supset G(J)$  and for every $J \in \mc{L}(\kappa)$ we have $G(J) \in \mc{L}(\ov{\kappa})$ (because $\kappa$ is pseudo-affine). If $z \in \dot{F}_nM$ is such that $\sigma(z) = \ov{z}$ then $\ov{J}\ov{z} = 0$ for some $\ov{J} \in \mc{L}(\ov{\kappa})$ yields $G(J)\ov{z} = 0$ for some $J \in \mc{L}(\kappa)$. Then $\wt{I}\wt{z} \subset \Ker \pi_{\wt{M}}$ for some $\wt{I} \in \mc{L}(\wt{\kappa}), \wt{I}\wt{z} \subset X\wt{M}$. If $ i \in \dot{F}_\gamma I, I \in \mc{L}(\kappa)$, then $\wt{i}\wt{z} = iz X^{\deg i + \deg z} = iz X^{\gamma + n}$ and $izX^{\gamma + n} = X\wt{m}$ for some $\wt{m} \in \wt{M}_{\gamma + n - 1}$ only if $iz \in F_{\gamma + n - 1}M$. The latters holds for all $\gamma$ and for all $i \in F_\gamma I$, thus we have $F_\gamma I z \subset F_{\gamma + n - 1}M$. The $\kappa$-separatedness of $M$ then yields that $z \in \kappa(M)$ and finally $\ov{z} \in G(\kappa(M))$.\\
The exact sequence of strict filtered morphisms:
$$ 0 \to \kappa(M) \to M \to M / \kappa(M)$$
yields the exact sequence
$$0 \to G(\kappa(M)) \to G(M) \to G(M/\kappa(M)) \to 0$$
and from $G(\kappa(M)) = \ov{\kappa}(G(M))$ it follows that $G(M/\kappa(M)) = \frac{G(M)}{\ov{\kappa}G(M)}$.
\end{proof}
\begin{corollary}
With assumptions as in f lemma we have
\begin{enumerate}
\item  If $M$ is $\kappa$-separated then $G(M/ \kappa(M))$ is $\ov{\kappa}$-torsion free.
\item If $G(M)$ is $\ov{\kappa}$-torsion free then $M$ is $\kappa$-separated.
\end{enumerate}
\end{corollary}
\begin{proof}
1. is obvious from foregoing lemma.We prove 2; Suppose $m \in \dot{F}_n M$ is such that $F_\gamma I m \subset F_{n + \gamma - 1}M$ for all $\gamma$ and some $I \in \mc{L}(\kappa)$. Put $\ov{m} = \sigma(m) \in G(M)_n$. Then clearly $\wt{I}\wt{m} \subset X\wt{M}$ with $\wt{I} \in \mc{L}(\wt{\kappa})$, then $\pi_{\wt{M}}(\wt{I}\wt{m}) = 0$ or $\pi(\wt{I})\ov{m} = 0$. Since $\ov{\kappa}G(M) = 0$ and $\pi(\wt{I}) \in \mc{L}(\ov{\kappa})$, we must have $\ov{m} = 0$, contradicting $ m \in \dot{F}_n M$.
\end{proof}
\begin{corollary}
With assumptions as before, If $M$ is $\kappa$-separated then $M/ \kappa(M)$ is $\kappa$-separated.
\end{corollary}
\begin{proof}
$G(M/ \kappa(M)) = G(M) / \ov{\kappa} G(M)$ is $\ov{\kappa}$-torsion free thus the filtered $R$-module $M/\kappa(M)$ (with induced filtration by $FM$) is $\kappa$-separated.
\end{proof}

\begin{theorem}\thelabel{locfilt2}
\begin{enumerate}
\item Let $FR$ be Zariskian, $M \in R$-filt is separated and $\wt{\kappa}$ is pseudo-affine. If $M$ is $\kappa$-separated (or $M/\kappa(M)$ is $\kappa$-separated in the affine case) then the localized $Q_\kappa(M)$ has a filtration $FQ_\kappa(M)$ making the localization morphism $j_{\kappa,M}: M \to Q_\kappa(M)$ into a strict filtered morphism (similar for $FR$ when $R$ is $\kappa$-separated, then $j_\kappa: R \to Q_\kappa(R)$ is strict filtered and a ring morphism). We call $FQ_\kappa(M)$ (resp. $FQ_\kappa(R)$) the quotient filtration of $FM$ (resp. $FR$). Moreover $Q_\kappa(M)$ with $FQ_\kappa(M)$ is a filtered $Q_\kappa(R)$-module with respect to $FQ_\kappa(R)$. 
\item In the situation of 1. with $\wt{\kappa}$ pseudo-affine, $Q_\kappa(M)^\sim = Q^g_{\wt{\kappa}}(\wt{M})$, where $Q^g_{\wt{\kappa}}$ is the graded localization functor associated to $\wt{\kappa}$ on $\wt{R}$-gr. 
\item In the situation of 2., we also have $G(Q_\kappa(M)) \subset Q_{\ov{\kappa}}^g(G(M))$, where $Q^g_{\ov{\kappa}}$ is the graded localization functor w.r.t. $\ov{\kappa}$ on $G(R)$-gr.
\end{enumerate}
\end{theorem}

\begin{proof}
The proof of 1. is completely analogous to the proof of \theref{locfilt}.\\
2) First we check that $Q^g_{\wt{\kappa}}(\wt{M})$ is $X$-torsion free. It follows from \lemref{tilde} that
$\wt{M}/\wt{\kappa}(\wt{M}) = (M/\kappa(M))^\sim$, hence the reduction to $M$ $\kappa$-torsion free also leads to $\wt{M}$ being $\wt{\kappa}$-torsion free. Now $\wt{M}/ \wt{\kappa}(\wt{M})$ is $X$-torsion free (in fact it is $\wt{M}$ now). If $q$ is homogeneous in $Q^g_{\wt{\kappa}}$ such that $Xq = 0$ then there is a $\wt{J} \in \mc{L}(\wt{\kappa})$ such that $\wt{J}q \subset \wt{M}$ and also $X\wt{J}q = \wt{J}Xq = 0$ hence $\wt{J}q = 0$ since $\wt{M}$ is $X$-torsion free. Then $\wt{J}q = 0$ contradicts the fact that $Q^g_{\wt{\kappa}}(\wt{M})$ is $\wt{\kappa}$-torsion free. By 1. the exact sequence
$$ 0 \to M \to Q_\kappa(M) \to Q_\kappa(M) / M = T \to 0$$
is strict exact, so we arrive at an exact sequence in $\mc{F}_X$:
$$0 \to \wt{M} \to Q_\kappa(M)^\sim \to \wt{T} \to 0$$
with $\wt{T}$ being $\wt{\kappa}$-torsion ($t \in T$, then $It = 0$ for some $I \in \mc{L}(\kappa), \forall i \in I, it = 0$ yields $\wt{i}\wt{j} = itX^{\deg it}X^{\deg i + \deg t - \deg it} = 0$). Thus $Q_\kappa(M)^\sim \subset Q^g_{\wt{\kappa}}(\wt{M})$. The latter is $X$-torsion free. Let us check that $Q^g_{\wt{\kappa}}(\wt{M}) / Q_\kappa(M)^\sim$ is $X$-torsion free. Suppose that $\wt{q} \in Q^g_{\wt{\kappa}}(\wt{M})_n$ is such that $X\wt{q} \in Q_\kappa(M)^\sim$. Since $\wt{\kappa}$ is pseudo-affine we have for some $I \in \mc{L}(\kappa)$ that $\wt{I}\wt{q} \subset \wt{M}$ and $\wt{I}X\wt{q} \subset X\wt{M}$. Since $X\wt{q}$ is in $(Q_\kappa(M)^\sim_{n+1})$ we may write $X\wt{q} = mX^{n+1}$ with $m \in F_{n+1}Q_\kappa(M)$. If $m \in F_nQ_\kappa(M)$ then $\wt{q} = mX^n \in Q_\kappa(M)^\sim$ as desired. So look at the case $m \notin F_nQ_\kappa(M)$. From $\wt{i}\wt{q} \in \wt{M}_{\gamma +n}$ if $\wt{i} \in \wt{I}_\gamma$, it follows that $\wt{i}X\wt{q} \in \wt{M}_{\gamma + n + 1}$ and $\wt{i}X\wt{q} = \wt{i}(mX^{n+1}) = imX^{\gamma + n + 1}$, thus $\wt{i}\wt{q} = im X^{\gamma +n}$ since $Q^g_{\wt{\kappa}}(\wt{M})$ is $X$-torsion free. From $\wt{i}\wt{q} \in \wt{M}_{\gamma + n}$ we then must have $im \in F_{\gamma + n}M$ (note that $\wt{I}$ is actually the tilde of an $I$ and not just an ideal in $\mc{L}(\wt{\kappa})$). Thus for all $\gamma, F_\gamma I m \subset F_{\gamma + n}M$, and $m \in \dot{F}_{n+1}Q_\kappa(M)$. Then $\sigma(m) \in G(M)_{n+1}$  has $G(I)\sigma(m) = 0$, thus $\sigma(m) \in \ov{\kappa}(G(M))$. Since we reduced to the $\kappa$-torsion free case $\ov{\kappa}(G(M)) = G(\kappa(M)) (M  ~\kappa$-separated) hence $\ov{\kappa}G(M) = 0$, thus $\sigma(m) = 0$, a contradiction. Therefore $Q^g_{\wt{\kappa}}(\wt{M})/Q_\kappa(M)^\sim$ is an $X$-torsion free $\wt{R}$-module and thus $Q^g_{\wt{\kappa}}(\wt{M}) = \wt{N}$ for a filtered $R$-module $N$ such that $ 0 \to Q_\kappa(M) \to N$ is a strict filtered exact sequence. Since $Q^g_{\wt{\kappa}}(\wt{M})/Q_\kappa(M)^\sim$ is $\wt{\kappa}$-torsion, $N / Q_\kappa(M)$ is $\kappa$-torsion, because $(N/Q_\kappa(M))^\sim = Q^g_{\wt{\kappa}}(\wt{M})/Q_\kappa(M)^\sim ( \in \mc{F}_X \subset \wt{R}$-gr). Thus if $n \in N/Q_\kappa(M)$ then $\wt{I}\wt{n} = 0$ for some $\wt{I} \in \mc{L}(\wt{\kappa})$ and $\wt{i}\wt{n}  = (in)^\sim X^{\deg i + \deg n - \deg in} = 0$. Hence $(in)^\sim = 0$ because $(N/Q_\kappa(M))^\sim$ is $X$-torsion free. $N/Q_\kappa(M)$ being $\kappa$-torsion implies that $N/M$ is $\kappa$-torsion, hence $N \subset Q_\kappa(M)$  and therefore we arrive at $Q^g_{\wt{\kappa}}(\wt{M}) = Q_\kappa(M)^\sim$.\\
3) Since $G(\kappa(M)) = \ov{\kappa}G(M)$ follows from the $\kappa$-separatedness of $M$ we may also in this part restrict the problem to the $\kappa$-torsion free case, i.e. suppose $\kappa(M) = 0$ and $\ov{\kappa}G(M) = 0$. Consider $\ov{q} \in G(Q_\kappa(M))_n$, i.e. $\ov{q} = \sigma(q)$ for $q \in \dot{F}_nQ_\kappa(M)$. There is as an $I \in \mc{L}(\kappa)$ such that $F_\gamma I q \subset F_{n+q}M$ for all $\gamma \in \mathbb{Z}$. For $i \in \dot{F}_\gamma I$ we have that $\sigma(i)\sigma(q)$ is either 0, or else $\sigma(i)\sigma(q) = \sigma(iq) \in G(M)$, thus always $\sigma(i)\sigma(q) \in G(M)$ and thus $G(I)\sigma(q) \subset G(M)$ or $\sigma(q) = \ov{q}$ is in $Q_{\ov{\kappa}}(G(M))$. Consequently $G(Q_\kappa(M)) \subset Q_{\ov{\kappa}}(G(M))$.
\end{proof}
\begin{proposition}
In the situation of the theorem, if $\wt{\kappa}$ is perfect then $\ov{\kappa}$ is perfect and $G(Q_\kappa(R)) = Q_{\ov{\kappa}}(G(R)), G(Q_\kappa(M)) = Q_{\ov{\kappa}}(G(M))$.
\end{proposition}
\begin{proof}
Since $R$ and $\wt{R}$ are Noetherian, $\wt{\kappa}$ has finite type and as a graded kernel functor we then have $Q_{\wt{\kappa}}(\wt{M}) = Q^g_{\wt{\kappa}}(\wt{M})$. Hence $Q_{\wt{\kappa}}(\wt{M}) = Q_\kappa(M)^\sim$ and $GQ_\kappa(M) = Q_{\wt{\kappa}}(\wt{M}) /XQ_{\wt{\kappa}}(\wt{M})$. By perfectness of $\wt{\kappa}$: $Q_{\wt{\kappa}}(\wt{M}) /XQ_{\wt{\kappa}}(\wt{M})  = Q_{\wt{\kappa}}(\wt{M}/X\wt{M}) = Q_{\wt{\kappa}}(G(M))$. As $\wt{\kappa}$ induces $\ov{\kappa}$ on $G(R)$-mod (also on $G(R)$-gr), we have $Q_{\wt{\kappa}}(G(M)) = Q_{\ov{\kappa}}(G(M))$. From $Q^g_{\wt{\kappa}}(\wt{M}) = Q_\kappa(M)^\sim$ it then follows that $G(Q_\kappa(M)) = Q_{\ov{\kappa}}(G(M))$.\\
Concerning the perfectness of $\ov{\kappa}$, take $L \in \mc{L}(\ov{\kappa})$. Then $L \supset G(I)$ for some $I \in \mc{L}(\kappa)$. Look at $Q_{\ov{\kappa}}(G(R))G(I), \pi(Q_{\wt{\kappa}}(\wt{R})\wt{I}) = G(Q_\kappa(R)) G(I) \subset Q_{\ov{\kappa}}(G(R))G(I)$. Now by the perfectness of $\wt{\kappa}$, $Q_{\wt{\kappa}}(\wt{R})\wt{I} = Q_{\wt{\kappa}}(\wt{R})$ hence $Q_{\ov{\kappa}}(G(R))G(I) = Q_{\ov{\kappa}}(G(R))$ holds for all $I \in \mc{L}(\kappa)$ and then for all $L \in \mc{L}(\ov{\kappa})$ we have that $Q_{\ov{\kappa}}(G(R))L = Q_{\ov{\kappa}}(G(R))$. This means $\ov{\kappa}$ is perfect.
\end{proof}

\begin{remark}\remlabel{remarkOre}
If $S$ is an Ore set of $R$ then even if $\sigma(S)$ is an Ore set of $G(R)$ we may not have that $\ov{\kappa}_S = \kappa^g_{\sigma(S)}$; indeed $G(Rs) \neq G(R)\sigma(s)$ is possible (we always have $\ov{\kappa}_S \leq \kappa^g_{\sigma(S)}$). However if $G(R)$ is $\sigma(S)$-torsion free then $\sigma(r)\sigma(s) = \sigma(rs)$ for all $r \in R$ and $G(Rs) = G(R)\sigma(s)$ and $\ov{\kappa}_S = \kappa^g_{\sigma(s)}$ follow. So when $R$ is a Zariskian ring with $G(R)$ being a domain then $R$ is $\kappa$-separated for all $\kappa$, and localization at Ore sets behaves very nice (e.g. $\wt{S}$ is an Ore set in $\wt{R}$ then too).
\end{remark}
\begin{proposition}
In the situation of the previous proposition, $\kappa$ is perfect too.
\end{proposition}
\begin{proof}
Consider $I \in \mc{L}(\kappa)$ and $Q_\kappa(R)I$ (we may reduce to the case where $R$ and $M$ are $\kappa$-torsion free as before). Since $\wt{\kappa}$ is perfect $Q_\kappa(R)^\sim = Q_{\wt{\kappa}}^g(\wt{R}) = Q_{\wt{\kappa}}^g(\wt{R})\wt{I}$, hence $1 = \sum_{j=1}^{'}\wt{q}_j \wt{i}_j$ with $\wt{q}_j \in Q_{\wt{\kappa}}^g(\wt{R}), \wt{i}_j \in \wt{I}$ homogeneous and we have $\wt{q}_j \wt{i}_j \in Q^g_{\wt{\kappa}}(\wt{R})_0$, i.e. $\wt{q}_j \in Q^g_{\wt{\kappa}}(\wt{R})_{-n}, \wt{i}_j \in Q^g_{\wt{\kappa}}(\wt{R})_n$ for some $n \in \mathbb{Z}$. Then $\wt{q}_j = q_jX^{-n}$ with $q_j \in F_nQ_\kappa(R), \wt{i}_j = i_jX^n$ with $i_j \in F_n I$ and $1 = \sum_{j=1}^{'} q_ji _j X^0 = \sum_{j=1}^{'} q_ji_j\in Q_\kappa(R)I$. From $Q_\kappa(R)I = Q_\kappa(R)$ for all $I \in \mc{L}(\kappa)$ we have that $\kappa$ is perfect.
\end{proof}

\section{Sheaves of Glider Representations}

In this final section we establish some sheaf theories of glider representations. Simultaneously, we obtain some sheaves of filtered modules, by which we mean an ordinary sheaf of modules, such that on every open set, the sections give a filtered module and the restriction morphisms are also filtered. Concerning the glueing axiom, the glueing element $x$ of a compatible set of sections $x_i$ must have a degree not exceeding the highest appearing degree of the $x_i$. Before we get to this however, we need some additional results. We assume either that all kernel functors $\kappa$ occurring below are coming from some $\wt{\kappa}$, or else we assume that $G(R)$ is a domain and $G(M)$ is a faithful $G(R)$-module. The crucial property we will need is namely that $G(\kappa(\Omega)) = \ov{\kappa}(G(\Omega))$, see \lemref{2} and \lemref{sep2}. First we observe that the filtered morphism $\rho^\kappa_\tau$ for $\kappa \leq \tau$ from \propref{filteredmor} restricts to $F_0Q_\kappa(\Omega) \to F_0Q_\tau(\Omega)$, hence to $Q_\kappa(M) \to Q_\tau(M)$ for a glider $M \subset \Omega$. Additionally, all this remains valid for left exact preradicals.

\begin{remark}
By defining $Q_\kappa(M)$ via $Q_\kappa(\Omega)$ we obtain a notion of localization of a glider representation $M \subset \Omega$ which depends on $\Omega$, so correctly we should adopt the notation $Q_\kappa(M,\Omega)$. However, we will not do this for reason of simplicity. There remains a question therefore, how does $Q_\kappa(M,\Omega)$ depend on $\Omega$? Since $RM \subset \Omega$, and $RM$ in $\Omega$ are not easily related at first sight, this remains nontrivial. We postpone this aspect to forthcoming work. So we write $Q_\kappa(M)$ for $Q_\kappa(M,\Omega)$ when $\Omega$ is fixed.
\end{remark}

\begin{lemma} \lemlabel{sep3}
Let $\Omega$ be a filtered $R$-module and let $\kappa_i$ be kernel functors such that $\Omega$ is $\kappa_i$-separated. Then $\Omega$ is $\wedge \kappa_i$-separated.
\end{lemma}
\begin{proof}
Take $x \in \dot{F}_n\Omega$ and suppose there is an $I \in \mc{L}(\wedge \kappa_I)$ such that $F_\gamma I x \subset F_{\gamma + n - 1}\Omega$ for all $\gamma$. Since $I \in \mc{L}(\kappa_i)$ for all $i$, we get that $x \in \kappa_i\Omega$ for all $i$. Hence $x \in \wedge \kappa_i (\Omega)$.
\end{proof}
\begin{lemma}
Let $\tau \geq \kappa$ and $\Omega \in R$-filt such that $\Omega$ is $\tau, \kappa$-separated. Then $\Omega/ \kappa\Omega$ is $\tau$-separated.
\end{lemma}
\begin{proof}
From $(\Omega / \kappa(\Omega)) / (\tau(\Omega)/\kappa(\Omega)) = \Omega/\tau(\Omega)$ we have
$$\frac{ G(\Omega/\kappa(\Omega))}{G(\tau(\Omega)/\kappa(\Omega))} = G(\Omega/\tau(\Omega)) = G(\Omega)/ \ov{\tau}G(\Omega)$$
because $G(\tau(\Omega)) = \ov{\tau}G(\Omega)$ since $\Omega$ is $\tau$-separated and because $\tau(\Omega)$ has by definition the induced filtration from $\Omega$. Thus $G((\Omega/\kappa(\Omega))/\tau(\Omega/\kappa(\Omega)))$ is $\ov{\tau}$-torsion free, hence $\Omega/\kappa(\Omega)$ is $\tau$-separated.
\end{proof}
\begin{remark}
There is a direct proof via the $\wt{\kappa}, \wt{\tau}, \kappa \leq \tau$. Although the proof is longer, it highlights the property that the gradation yields $(\wt{J}\wt{I})_{\tau + \gamma} = \sum_{\tau' + \gamma' = \tau + \gamma}\wt{J}_{\tau'} \wt{I}_{\gamma'}$. We omit the proof here.
\end{remark}

Let $M \subset \Omega$ be a glider representation and let $\kappa_i$ be a finite number of kernel functors (preradicals) such that $\Omega
$ is $\kappa_i \wedge \kappa_j$-separated for all $i,j$. Since $\wedge \kappa_i \leq \kappa_i$,
$$Q_{\wedge \kappa_i}(\Omega) \mapright{\rho_i} Q_{\kappa_i}(\Omega)$$
is filtered, so we can restrict to the degree 0 part
$$Q_{\wedge \kappa_i}(M) \mapright{\rho_i} Q_{\kappa_i}(M)$$
\begin{proposition}\proplabel{cover}
With the above assumptions we have that if $x_i \in F_0Q_{\kappa_i}(\Omega) = Q_{\kappa_i}(M)$ are such that $\rho^i_{ij}(x_i) = \rho^j_{ij}(x_j)$ and if there exists an $x \in Q_{\wedge \kappa_i}(\Omega)$ such that $ \rho_i(x) = x_i$ for all $i$, then $x \in F_0Q_{\wedge \kappa_i}(\Omega) = Q_{\wedge \kappa_i}(M)$.
\end{proposition}
\begin{proof}
Suppose $x \in \dot{F}_nQ_{\wedge \kappa_i}(\Omega)$ for some $n > 0$. There exists $I \in \mc{L}(\wedge \kappa_i)$ such that for all $\gamma$~: $F_\gamma I x \subset F_{n + \gamma}(\Omega / \wedge\kappa_i \Omega)$ and for some $\gamma$, $F_\gamma I x \not\subset F_{n+ \gamma - 1}(\Omega/ \wedge\kappa_i \Omega)$. Then $\rho_i (F_\gamma I x) = F_\gamma I x_i \subset F_{n+\gamma}(\Omega/ \kappa_i \Omega)$. Since $I \in \mc{L}(\kappa_i)$ for every $i$, $F_\gamma I x_i \subset \dot{F}_{n+\gamma}(\Omega / \kappa_i \Omega)$ would contradict that $x_i \in F_0Q_{\kappa_i}(\Omega)$ ($\Omega/\kappa_i \Omega$ is $\kappa_i$-separated so $\deg x_i$ does not depend on the ideal in $\mc{L}(\kappa_i)$ used to define it!). Consequently $F_\gamma I x_i \subset F_\gamma(\Omega / \kappa_i \Omega)$ for all $\gamma$, i.e.
$$F_\gamma I x_i \subset F_\gamma(\Omega / \wedge \kappa_j \Omega) + \kappa_i(\Omega / \wedge \kappa_j \Omega).$$
Since $F_\gamma I x \subset F_{n+\gamma} (\Omega / \wedge \kappa_i \Omega)$ with $n > 0$ we obtain
\begin{eqnarray*}
F_\gamma I x &\subset&( F_\gamma(\Omega / \wedge \kappa_i \Omega) + \kappa_i(\Omega / \wedge \kappa_i \Omega)) \cap F_{n+ \gamma}(\Omega / \wedge \kappa_i \Omega)\\
&\subset& F_\gamma(\Omega / \wedge \kappa_i \Omega) + F_{n+ \gamma}\kappa_i(\Omega / \wedge \kappa_i \Omega),
\end{eqnarray*}
hence for $i_\gamma \in F_\gamma I$ with $i_\gamma x \notin F_{n + \gamma - 1}(\Omega / \wedge \kappa_i \Omega)$:
$$i_\gamma x \mod F_{n + \gamma - 1} (\Omega / \wedge \kappa_i \Omega) \in \frac{F_{n + \gamma}\kappa_i(\Omega / \wedge \kappa_i \Omega)}{F_{n+ \gamma - 1} \kappa_i(\Omega / \wedge \kappa_i \Omega)}.$$
The latter means that $\ov{i_\gamma x} \in G_{n+ \gamma}(\kappa_i(\Omega / \wedge \kappa_i \Omega)) = \ov{\kappa_i}G_{n+\gamma}(\Omega / \wedge \kappa_i \Omega)$, where the last equality follows from \lemref{sep2}. This holds for all $i$, hence $\ov{i_\gamma x} \in \ov{\wedge \kappa_i}G(\Omega / \wedge \kappa_i \Omega)$ but as $\Omega / \wedge \kappa_i \Omega$ is $\wedge \kappa_i$ separated by \lemref{sep3} we have that $\ov{\wedge \kappa_i}G(\Omega / \wedge \kappa_i \Omega) = 0$. But then $i_\gamma x \in F_{n+\gamma - 1}(\Omega / \wedge \kappa_i\Omega)$, a contradiction. Hence $x \in F_0Q_{\wedge \kappa_i}(\Omega) = Q_{\wedge \kappa_i}(M)$.
\end{proof}

\begin{remark}
In fact, one analogously proves a slightly more general statement, saying that if $x_i \in F_mQ_{\kappa_i}(\Omega)$ for some $m \in \mathbb{Z}$ are compatible on the intersections, then if a glueing element $x$ exists, then it must be in $F_mQ_{\wedge \kappa_i}(\Omega)$. This proves for instance that if an $R$-module $\Omega$ yields a sheaf for some topology (e.g. the classical $\wt{\Omega}$ if $R$ is commutative), then a filtration such that $G(\Omega)$ is faithful as $G(R)$-module, yields a sheaf of filtered modules. In particular, if $FR$ is `almost commutative', then we obtain a filtered structure sheaf $\mc{O}_X$ for $X$ some suitable topological space (e.g. the Zariski topology). In this case, we denote the filtered sheaf by $\mc{FO}_X$.
\end{remark}
The previous proposition allows to define sheaves of glider representations for various topological spaces.
\begin{definition}\deflabel{glid}
Let $X$ be some topological space and $FR$ a filtered ring such that the structure sheaf $\mc{FO}_X$ is filtered. A sheaf $\mc{F}$ of glider representations is a sub-presheaf of a presheaf $\mc{G}$ of filtered $\mc{FO}_X$-modules such that for every open set $U \subset X$, $\mc{F}(U) = F_0\mc{G}(U)$ and such that $\mc{F}$ satisfies the separability and glueing axioms.
\end{definition}
Assume from now on that $FR$ is an `almost commutative' Zariskian ring. In particular, we know that $R$ is $\kappa$-separated for every kernel functor (or preradical) on $R$-mod.\\

First we consider the Zariski topology $\Spec(R)$ for a Noetherian prime ring (see \cite{MVo} for its construction). There is a basis of the Zariski topology consisting of open sets $X(I)$ where $I \triangleleft R$ is an ideal. The kernel functor $\kappa_I$ associated to an ideal $I$ is determined by the filter
$$\mc{L}(\kappa_I) = \{ {\rm left~ideals~} L {\rm~of~} R {\rm~such~that~} L \supset I^n {\rm~for~some~positive~integer~}n\},$$
and $\kappa_I$ is symmetric. Let $M$ be a left $R$-module, which is $\kappa_I$-torsion free for all ideals $I$.  Assigning $Q_{\kappa_I}(M) = Q_I(M)$ to the open set $X(I)$ defines a sheaf $\wt{M}$ of $R$-modules on $\Spec(R)$ (cf. \cite[Theorem 14]{MVo}). Since $X(I) \subset X(J)$ if and only if $\kappa_J \leq \kappa_I$, we have by \propref{filteredmor} that the restriction map
$$Q_I(M) \to Q_J(M)$$
is filtered. Moreover, if $X(I) = \cup X(I_j)$ is a finite cover, then it is shown in \cite[Theorem 13]{MVo} that $\cap \kappa_{I_j}(N) = \kappa_I(N)$ for all $N \in R$-mod. Hence $\mc{T}_{\kappa_I} = \cap \mc{T}_{\kappa_{I_j}} = \mc{T}_{\wedge \kappa_{I_j}}$ and $\kappa_I = \wedge \kappa_{I_j}$ follows.\\
If $M \subset \Omega$ is a glider representation, we denote by $\wt{M}$ the presheaf obtained by restricting $\wt{\Omega}$ to the degree zero part, that is, $\Gamma(X(I),\wt{M}) = F_0Q_I(\Omega) = Q_I(M)$. We denote by $\mc{F}^+\wt{R}$ the presheaf given by $\Gamma(X(I),\mc{F}^+\wt{R}) = F^+Q_I(R)$.

\begin{proposition} 
Let $R$ be a Zariskian prime ring (hence Noetherian) and $M$ a left $R$-module such that $G(M)$ is a faithful $G(R)$-module. Then $\wt{M}$ becomes a sheaf of filtered $\mc{F}\wt{R}$-modules. If $M \subset \Omega$ is glider such that $G(\Omega)$ is a faithful $G(R)$-module, then $\wt{M}$ is a sheaf of glider $\mc{F}^+\wt{R}$-representations.
\end{proposition}
\begin{proof}
For the separated presheaf condition, we refer to \propref{sepa} where it is shown in a more general setting. Observe that we need that $\kappa_I = \wedge \kappa_{I_j}$ if $X(I) = \cup X(I_j)$ is a cover. The glueing axiom follows directly from \propref{cover}. 
\end{proof}

\begin{remark}
Since $G(R)$ is commutative, we have $Y = \Proj( G(R))$, the classical Zariski topology. In \cite[Section 2.4.A]{VoAG} it is shown that 
\begin{equation}
\left\{ \begin{array}{l}
\kappa_{<P>} = \vee \{\kappa_{<f>}, \ov{f} \notin P \} = \vee \{ \kappa_{<I>}, \ov{I} \not\subset P\}\\
\kappa_{<f>} = \wedge \{ \kappa_{<P>}, \ov{f} \notin P\}\\
\kappa_{<I>} = \wedge \{ \kappa_{<P>}, \ov{I} \not\subset P\} = \wedge_{i=1}^d \kappa_{<f_i>}, {\rm~if~} \ov{I} = (\ov{f}_1,\ldots, \ov{f}_n)
\end{array} \right.
\end{equation} 
with notations as in loc. cit. By deriving sheaf-conditions on the level of torsion theories, one obtains sheaves $\un{M}_Y$ for arbitrary $R$-modules $M$ (with no torsion freeness assumptions!) (cf. Theorem 2.4.5 and 2.5.6 in loc. cit.). In order to let everything work in our filtered setting, it suffices to have a basis of the Zariski topology for which the filtered module $M$ is separated. The filtration on a general open set can then be defined by an inverse limit, cf. \cite[Observation before Remark 2.6, p.8]{LiVo}.
\end{remark}

Next, we consider $R$-tors, the set of all kernel functors on $R$-mod, which has a (distributive) lattice structure (see \cite{Gol} for a detailed overview). We define a topology on $R$-tors by defining a basis consisting of the open sets $\gen(\tau) = \{ \tau' \in R-{\rm tors}, \tau \leq \tau'\}$. We call this the gen-topology of $R$-tors. Recall that $\tau \leq \tau'$ if $\mc{L}(\tau) \subset \mc{L}(\tau')$. We have 
$$\gen(\kappa) \cap \gen(\tau) = \gen(\kappa \vee \tau)$$
$$\gen(\kappa) \cup \gen(\tau) = \gen(\kappa \wedge \tau)$$
The distributivity of the lattice implies that
$$\gen(\kappa) \cup (\gen(\sigma) \cap \gen(\tau)) = (\gen(\kappa) \cup \gen(\sigma)) \cap (\gen(\kappa) \cup \gen(\tau)).$$
To a basic open set $U = \gen(\tau)$ we associate the localized ring $Q_{\wedge U}(R) = Q_\tau(R)$ which is filtered by \theref{locfilt}. Since $\gen(\tau) \subset \gen(\kappa)$ if and only if $\kappa \leq \tau$, the restriction morphism $Q_\kappa(R) \to Q_\tau(R)$ is also filtered. Hence we obtain a presheaf $\mc{F}\ov{Q}(-,R)$ of filtered rings. The same result holds for a filtered $R$-module $N$ such that $G(N)$ is faithful as $G(R)$-module, i.e. then $\ov{Q}(-,N)$ is a presheaf of filtered $\mc{F}\ov{Q}(-,R)$-modules. For $M \subset \Omega$ a glider representation, we can restrict the sections $\ov{Q}(\gen(\tau),\Omega) = Q_\tau(\Omega)$ to the degree zero part $F_0\ov{Q}(\gen(\tau),\Omega)$. We denote the latter by $\ov{Q}(\gen(\tau),M)$ and we obtain in this way a presheaf of $\mc{F}^+\ov{Q}(-,R)$ glider representations. We also use the notations $\mc{FO}_X, \mc{O}_\Omega, \mc{O}_M$ for the respective presheaves. 
 
\begin{proposition}\proplabel{sepa}
Let $M \subset \Omega$ be a glider representation such that $G(\Omega)$ is faithful as $G(R)$-module. If $\ov{Q}(-,\Omega)$ is a sheaf of filtered $\mc{F}\ov{Q}(-,R)$-modules we obtain a sheaf $\ov{Q}(\_,M)$ of glider $\mc{F}^+\ov{Q}(\-,R)$-representations. 
\end{proposition}
\begin{proof}
We first have to check the separated presheaf condition, so consider $x \in Q_{\wedge \kappa_i}(M)$ such that $\rho_i(x) = 0$ for every $i$ in the covering of $\wedge \kappa_i$ by the $\kappa_i$. Take $I \in \mc{L}(\wedge\kappa_i)$ such that $F_\gamma I x \subset F_{\gamma + n} M/\wedge \kappa_i M$. Then $\rho_i(F_\gamma I x) = 0$ means $F_\gamma I x \subset \kappa_i (M)/ \wedge \kappa_i M)$. This holds for all $i$, so $F_\gamma I x \subset \wedge \kappa_i M/\wedge \kappa_i M = 0$. This holds for all $\gamma$, i.e. $I x = 0$, but then $x \in (\wedge \kappa_i)(Q_{\wedge\kappa_i}(M)) = 0$.\\
For the glueing condition, letting $\gen(\kappa) = \cup_i \gen(\kappa_i)$ be a finite cover, meaning that $\kappa = \wedge \kappa_i$ (see \cite[Proposition 8.4]{Gol}) we may apply \propref{cover} directly.
\end{proof}

As was the case for $\Spec(R)$, we do not need to restrict to faithful modules. Indeed, it suffices again to have a basis $\gen(\kappa)$ for which $M$ is $\kappa$-separated and use inverse limits to define filtrations.\\

Recall that on $\wt{R}$-grtors we have an open set given by $\gen(\wt{\kappa}_X)$ of the affine torsion theories. Under the equivalence of categories $\wt{R}_X$-gr $\cong R$-mod, a filtered module $M$, corresponds to $\wt{M}_X$. Let $\wt{\kappa} \geq \wt{\kappa}_X$ be an affine kernel functor with associated $\kappa$ on $R$-mod. Since localizing at $X$ is perfect and since $\mc{L}(\wt{\kappa})$ contains $\wt{R}X^n$ for all $n\in \mathbb{N}$, we obtain that $\wt{\kappa}(\wt{M}_X) = \wt{\kappa}(\wt{M})_X$. Therefore
$$Q_{\wt{\kappa}}^g(\wt{M}_X) = Q_{\wt{\kappa}}^g(\wt{M})_X = (Q_\kappa(M)^\sim)_X,$$
which corresponds to the $R$-module $Q_\kappa(M)$ with quotient filtration. Hence the sheaf $\mc{O}^g_{\wt{M}}$ restricted to the open set $\gen(\wt{\kappa}_X) \subset \wt{R}$-grtors is isomorphic to the sheaf $\mc{O}_M$ on $R$-tors. Similarly, for $M \subset \Omega$ a glider representation, we can restrict to the degree zero part and obtain a similar result for sheaves of glider fragments over $R$-tors.\\

Finally, we discuss the noncommutative site from \cite{VoAG} in the filtered case. We give a concise survey of how this site is built. We refer the reader to loc. cit. for a detailed treatment.
 \begin{definition}\deflabel{schematic}
A noncommutative ring $R$ is said the be affine schematic if there exists a finite set of nontrivial Ore sets of $R$, say $S_1,\ldots,S_n$ such that for every choice of $s_i \in S_i, i=1,\ldots,n$, we have that $\sum_{i=1}^n Rs_i = R$ or equivalently: $\cap_i \mc{L}(S_i) = \{R\}$, in which $\mc{L}(S)$ is the Gabriel filter associated to the Ore set $S$ 
\end{definition}
In torsion theoretic language the latter means that the infimum of the kernel functors $\kappa_{S_i}$ associated to $S_i$ equals the trivial kernel functor, i.e. $\wedge_{i=1}^n \kappa_{S_i} = \kappa_1$ ($\kappa_1$ associated to the trivial Ore set $\{1\}$). 
One considers the free monoid $\mc{W}(R)$ on $\theta(R)$, the set of left Ore sets $S$ of $R$ and one introduces the category $\un{\mc{W}}(R)$ with objects the elements of $\mc{W}(R)$ and concerning morphisms: if $W = S_1\ldots S_n, W' = T_1 \ldots T_m$, then $\Hom(W',W) = \{ W' \to W\}$ or $\emptyset$ depending whether there exists a strictly increasing map $\alpha: \{1,\ldots, n\} \to \{1, \ldots, m\}$ for which $S_i = T_{\alpha(i)}$. Let $\{W_i,~ i \in I\}$ be a finite subset of objects of $\un{\mc{W}}$ such that $\cap_{i \in I} \mc{L}(W_i) = \mc{L}(\kappa_1)$, in which $\mc{L}(W) = \{L,~ w \in L {\rm~for~some~} w \in W\}$ is the associated filter of left ideals. We call such a subset a global cover of $Y = \un{\mc{W}}(R)$. Observe that the schematic property exactly states there exists such a cover. For $W \in \un{\mc{W}}$ let $\Cov(W)$ be $\{W_iW \to W,~ i \in I\}$. Then one shows that we obtain a (noncommutative) Grothendieck topology on this category.\\

Let $W=S_1\ldots S_n$ then we have associated $\mc{L}(\kappa_{S_1}\kappa_{S_2}\ldots\kappa_{S_n})$ as in the foregoing section to $W$.  By choosing $s_j$ as generators for $Rs_j$ we have that $Rs_1 \cdot Rs_2 \cdot \ldots \cdot Rs_n = Rs_1s_2\ldots s_n$, but one may take different generators, yielding different left ideals in the filter. In \cite{VoAG} however, one considered $\mc{L}(W) = \{Rs_1\ldots s_n, {\rm~with~} s_i \in S_i, i=1,\ldots, n\}$. Under some torsion freeness assumptions, we have that both filters are the same.
\begin{lemma}
Assume that $G(R)$ is a domain. We have $\mc{L}(\kappa_{S_1}\kappa_{S_2}\ldots\kappa_{S_n}) = \mc{L}(W)$.
\end{lemma}
\begin{proof}
For $Rs_i \in \mc{L}(\kappa_{S_i})$ we may choose another generator, $x$ say, $Rx = Rs_i$, thus $s_i = u_i x$ and $x = v_i s_i$. This yields $s_i = u_iv_i s_i$ and by $S$-torsion freeness of $R$, $1 = u_i v_i$. Since $G(R)$ is a domain, $R$ is a domain and then $u_i$ is invertible in $R$. As a filterbasis for $\mc{L}(\kappa_{S_1}\kappa_{S_2}\ldots\kappa_{S_n})$ we may take $Rs_1u_1s_2\ldots u_{n-1}s_n$ where the $u_j$ are invertible elements in $R$. By the Ore condition for $s_1$ we have $s_1'u_1^{-1} = r_1's_1$ for some $s_1' \in S_1, r_1' \in R$. Thus $s_1' = r_1's_1u_1$ and $Rs_1's_2u_2\ldots u_{n-1}s_n = Rr_1's_1u_1s_2u_2\ldots u_{n-1}s_n$, or $Rs_1's_2u_2\ldots u_{n-1}s_n \subset Rs_1u_1s_2u_2\ldots u_{n-1}s_n$. Now $R$ satisfies the second Ore condition with respect to $S_1S_2$ so $s_1''s_2''u_2^{-1} = r_2's_1's_2$ for some $s_1'' \in S_1, s_2'' \in S_2, r_2' \in R$. Hence $s_1''s_2'' = r_2's_1's_2u_2$ and therefore:
$$Rs_1''s_2''s_3u_3\ldots u_{n-1}s_n \subset Rs_1's_2u_2s_3\ldots u_{n-1}s_n \subset Rs_1u_1s_2u_2\ldots u_{n-1}s_n.$$
We repeat this procedure until we obtain $s_i^{(n)} \in S_i$ such that $Rs_1^{(n)}\ldots s_n^{(n)} \subset Rs_1u_1s_2 \ldots u_{n-1}s_n$, proving that $Rs_1u_1s_2\ldots u_{n-1}s_n \in \mc{L}(W)$ as desired.\\
Now let us look at $Rt, Rs$ for $s,t$ in Ore sets $S$ and $T$ resp. and we look at $Rt \cdot Rs$ corresponding to writing $Rs = Rus + Rvs$. So we have $Ru + Rv = R$, say $r_1u + r_2v = 1$. By the Ore conditions there are $t',t'' \in T$ and $x,y \in R$ such that $t'u = xt$ and $t''v = yt$. Take $t_1 \in Rt' \cap Rt'' \cap Rt$ then $t_1u = x't$ and $t_1v = y' t$ for some $x', y' \in R$. Again by the Ore conditions, $t_2r_1 = r_1't_1, t_3r_2 = r_2't_1$ and we may take again $t''' \in Rt_2 \cap Rt_3$ and obtain $t'''r_1 = r_1'' t_1$ and $t''' r_2 = r_2'' t_1$ for some $r_1'', r_2'' \in R$. Now from $r_1u + r_2v = 1$ we obtain that $t''' = t''' r_1 u + t''' r_2 v = r_1''t_1 u + r_2'' t_1 v$. Thus
$$t''' s = r_1''t_1 u s + r_2'' t_1 v s.$$
Since we took $t_1 \in Rt$ we have $t''' s = r_1''' tus + r_2''' tvs$ for some $r_1''', r_2''' \in R$, i.e. $t'''s \in Rt \cdot Rs$ with respect to the chosen generators $us, uv$ for $Rs$. The extension to more Ore sets in the word considered can easily be obtained by induction since the Ore property holds with respect to every word $S_1\ldots S_n$. For example, add another Ore set $U$ and consider $Ru\cdot Rt \cdot Rs$ corresponding to writing $Rs = Ras + Rbs$ and $Rt = Rct + Rdt$. By induction there exist $t''' \in T, u''' \in U$ such that 
$$ t'''s = r_1tas + r_2tbs, \quad u''' t= r_3uct + r_4udt, \quad r_i \in R.$$
By the Ore condition we find $u_1,u_2 \in U, r_1',r_2' \in R$ such that $u_1r_1 = r_1'u'''$ and $u_2r_2 = r_2'u'''$. Take some $u_3 \in Ru_1 \cap Ru_2$, then $u_3t''' s \in Ru \cdot Rt \cdot Rs$ with respect to the chosen generators. The extension to more generators is easy but technical to write down, essentially it follows from taking some $Rs_0$ in a finite intersection $Rs_1 \cap \ldots \cap Rs_n$, we leave the details to the reader. So we have proved that $\mc{L}(W) = \mc{L}(\kappa_{S_1}\kappa_{S_2}\ldots\kappa_{S_n})$, the localization of $R$ at the Ore set generated by the Ore sets $S_1, \ldots, S_n$.\\
\end{proof}

In \cite{VoAG} one constructed the structure sheaf $\mc{O}_R$ as the functor $\un{\mc{W}}(R)^{op} \to R$-mod sending $W= S_1\ldots S_n$ to $Q_{S_n}\ldots Q_{S_1}(R)$. The $Q_{S_i}(R)$ are rings but the $Q_W(R)$ are in general not rings but nevertheless we have a multiplication defined by $Q_{S_n}\ldots Q_{S_1}(R) \times Q_{T_m}\ldots Q_{T_1}(R) \to Q_{S_n}\ldots Q_{S_1}Q_{T_m}\ldots Q_{T_1}(R)$ (recall that the $Q_{S_n}\ldots Q_{S_1}(R)$ may be seen as the $n$-fold tensor product over $R$ $Q_{S_n}(R) \ot_R \ldots \ot_R Q_{S_1}(R)$ for every $W= S_1\ldots S_n$). In fact $Q_W(R)$ is a an $R$-submodule of $Q_{S_n \vee \ldots \vee S_1}(R)$.\\

In \cite{VoAG} one also defined $\mc{O}_M$ for an $R$-module $M$ and where we called $\mc{O}_R$ a structure sheaf of rings by ``abus de language" we also call $\mc{O}_M$ a structure sheaf of modules over the noncommutative Grothendieck topology considered. Again $\mc{O}_M$ is defined as the functor $\un{\mc{W}}(R)^{op} \to R$-mod sending the object $W$ to $Q_W(M) = Q_{S_n}\ldots Q_{S_1}(M)$. So $Q_W(M)$ is not a $Q_W(R)$-module but there is for every $Q_T(R)$ a scalar multiplication $Q_T(R) \times Q_W(M) \to Q_{WT}(M)$, hence also $Q_W(R) \times Q_W(M) \to Q_{WW}(M)$. In particular $Q_W(M)$ is a $Q_{S_n}(R)$-module!

\begin{definition}
We will call a presheaf of $\mc{O}_R$-modules, say $\mc{M}$, an $\mc{O}_R$-module if there are operations $\varphi(T,W): Q_T(R) \times \mc{M}(W) \to \mc{M}(WT)$ which respect the additions and act associatively i.e. $\varphi(ST,W) = \varphi(S, WT)$.
\end{definition}

In the case where we consider a filtration on $R$ we will usually restrict to a Zariskian filtration, in fact even a positive filtration on a Noetherian ring. Then when we assume that $G(R)$ is a domain, $R$ is a Noetherian domain and it has a classical total quotient ring $Q(R)$ which is a simple Artinian ring. The multiplication considered in $\mc{O}_R$ is just obtained from the multiplication in $Q(R)$, i.e. $Q_T(R)$ and $Q_W(R)$ can be multiplied in $Q(R)$ and in particular in $Q_{TW}(R) \subset Q(R)$. Similar, a filtered $R$-module $M$ such that $G(M)$ is $G(R)$-faithful is then a submodule of $Q(R) \ot_R M$ and the scalar multiplications $Q_W(R) \times Q_T(M) \to Q_{TW}(M)$ may be viewed as carried out in $Q(R) \ot_R M$.

\begin{definition}\deflabel{she}
A filtered $\mc{O}_R$-module is a presheaf of filtered $\mc{O}_R$-modules such that the restriction morphisms in $R$-filt are strict filtered morphisms. In particular for a filtered $\mc{O}_R$-module $\mc{M}$ each $\mc{M}(S_1\ldots S_n)$ is an $FQ_{S_n}(R)$-filtered module.\\
An $\mc{O}_R$-glider $\mc{F}$ is a subpresheaf of a filtered $\mc{O}_R$-module $\mc{G}$ obtained by taking for every $W$, $\mc{F}(W) = F_0\mc{G}(W)$ (compare \defref{glid}).\\
In case $\mc{G}$ is a sheaf such that the glueing axiom holds in $F_n\mc{G}$ for every $n$ then we call $\mc{G}$ a filtered $\mc{O}_R$-module sheaf. Similarly when $\mc{F}$ is a sheaf with the glueing axiom holding with respect to $\mc{F}_n$ for every $n$, then $\mc{F}$ is said to be an $\mc{F}^+\mc{O}_R$-glider sheaf.
\end{definition}
In the remainder of this paper we restrict to a noncommutative geometry situation: $R$ is Noetherian and positively filtered by $FR$ such that $F_0R = K$ is a field, and $R$ is affine schematic (cf. \defref{schematic}). Since we aim to work with separated filtrations on the localizations of $R$ we assume moreover that $G(R)$ is a domain. In the book \cite{VoAG} there are given many examples of such schematic algebras and these algebras have moreover standard filtrations. E.g. Weyl algebras, enveloping algebras of Lie algebras, Sklyanin algebras, Witten gauge algebras, rings of differential operators on regular varieties, quantum-sl(2), many quantum groups obtained as iterated Ore extensions, etc.\\
In this case we consider $\un{\mc{W}}(R)$, the set of words in Ore sets in nontrivial Ore sets, that is $S \in \un{\mc{W}}(R)$ if $S \cap K = \{1\}$.
\begin{proposition}\proplabel{backandforth2}
In the situation considered above, we consider an $\mc{F}^+\mc{O}_R$-glider $\mc{F}$ and its filtered $\mc{O}_R$-module $\mc{G}$, such that $\mc{F}(W) = F_0\mc{G}(W)$ for all $W \in \un{\mc{W}}(R)$. Then $\mc{F}$ is an $\mc{F}^+\mc{O}_R$-glider sheaf if and only if $\mc{G}$ is a filtered $\mc{O}_R$-module sheaf.
\end{proposition}
\begin{proof}
We only have to check the separation and glueing axiom.\\ 
1. Assuming that $\mc{G}$ is a filtered $\mc{O}_R$-module sheaf, both the separation and glueing axiom are easily verified for $\mc{F}$ (using that the glueing axiom on $\mc{G}$ holds in fact in $F_n\mc{G}$ for every $n$).\\
2. Assume that $\mc{F}$ is an $\mc{F}^+\mc{O}_R$-glider sheaf. Look at a cover $\{ W \to W_iW\}$ and suppose $x \in \mc{G}(W)$ maps to 0 in each $\mc{G}(W_iW)$. If $W = S_1\ldots S_n$ then $\mc{G}(W)$ is a $Q_{S_n}(R)$-module and for every $s_n \in S_n$ we have $s_n^{-1} \in Q_{S_n}(R)$. Since $G(R)$ is a domain $\deg(s_n^{-1}) = - \deg(s_n) < 0$ since $S \cap K = \{1\}$. So we can choose $\deg(s_n)$ high enough so that $s_n^{-1}x \in F_0\mc{G}(W)$ and $s_n^{-1}x$ still maps to zero in each $\mc{G}(W_iW)$. The separation axiom in $\mc{F}$ then yields that $s_n^{-1}x = 0$, hence $x = 0$. So $\mc{G}$ is satisfying the separation axiom.\\
Consider a finite cover $\{W \to W_iW\}$ and look at $x_i \in \mc{G}(W_iW)$ such that $x_i$ and $x_j$ are mapping to the same element in $\mc{G}(W_iWW_jW)$ for every $i$ and $j$. Since $S_n$ is the last letter of $W, W_iW, W_iWW_jW$ for all $i$ and $j$, all sections are $Q_{S_n}(R)$-modules and again we may choose $s_n \in S_n$ such that $s_n^{-1}x_i \in F_0\mc{G}(W_iW) = \mc{F}(W_iW)$, say $s_n^{-1}x_i \in \mc{F}_d(W_iW)$ for some $d \leq 0$ (i.e. $x_i \in F_{d + \deg(s_n)}\mc{G}(W_iW)$) and we may assume $d$ is smallest possible. By the glueing property of $\mc{F}$ there is a $y \in \mc{F}(W)$ such that $y$ maps to $s_n^{-1}x_i$ under each restriction morphism, in fact since the glueing property of $\mc{F}$ respects the $\mc{F}_m$, we have $y \in \mc{F}_d(W)$. Then $s_ny$ maps to $x_i$ under each restriction from $W$ to $W_iW$, and $s_ny \in F_{d + \deg(s_n)}\mc{G}(W)$.\\
We still have to check the axioms for the trivial word $W=1$. Since $G(R)$ is a domain, we have a total quotient ring $Q(R)$ and we can consider $\mc{G}(W)$ as an $R$-submodule of $Q(R) \ot_R \mc{G}(1)$. If elements $x_i \in \mc{G}(W_i)$ are such that $x_i$ and $x_j$ map to the same element in $\mc{G}(W_iW_j)$ for every $i$ and $j$, then since the cover is finite, they become equal in $\mc{G}(W) \subset Q(R) \ot_R \mc{G}(1)$ for some word $W$. This implies that the $x_i$ are in the intersection of the $\mc{G}(W_i)$. But since $\{W_i\}$ is a global cover, this intersection is just $\mc{G}(1)$. So we established the filtered glueing property of $\mc{G}$.
\end{proof}
\begin{corollary}
Let $M \subset \Omega$ be a glider representation such that $G(\Omega)$ is a faithful $G(R)$-module (assumptions on $R, G(R)$ as before). The presheaves $\mc{O}_M$ and $\mc{O}_\Omega$, $\mc{O}_M = F_0\mc{O}_\Omega$ defined before \defref{she} are both sheaves if one is.
\end{corollary}
\begin{example}
For $FR$ commutative, say with $F_0R =K$ a field,  we have the classical Zariski topology $X = \Spec R$ with structure sheaf $\mc{O}_X$. If $G(R)$ is a domain, then $\mc{FO}_X$ is a filtered sheaf. The proof of the previous proposition reduces in this case to considering a finite cover $X = \bigcup_i\mathbb{X}(f_i)$, with nontrivial Ore sets $<f_i >$. For any $g \in R$, we have a cover $\mathbb{X}(g) = \cup_i\mathbb{X}(gf_i)$. Since $\mathbb{X}(gf_i) \subset \mathbb{X}(g)$ (if and only if $gf_i \in \sqrt{(g)}$) implies that $\Omega_{gf_i}$ is an $R_g$-module, we can lower the degree of $x_i \in \Omega_{gf_i}$ by multiplying with $g^{-1}$. 
\end{example}
\begin{definition}
A sheaf $\mc{F} = \mc{F}_0\mc{G}$ of $\mc{O}_R$-glider representations is quasi-coherent if there exists a global cover of Ore sets $\{T_i, i \in I\}$ together with $F^+Q_{T_i}(R)$ glider representations $M_i \subset \Omega_i$ such that for any morphism $V \to W$ in $\un{\mc{W}}(R)$ we obtain a commutative diagram of filtered morphisms:
\begin{equation} \label{dia}
\xymatrix{ & \mc{G}(T_iW) \ar[rr] \ar[dd] &&\mc{G}(T_iV) \ar[dd] \\
\mc{F}(T_iW) \ar@{^{(}->}[ru]  \ar[dd] \ar[rr] && \mc{F}(T_iV) \ar@{^{(}->}[ru]  \ar[dd] &\\
& Q_W(\Omega_i) \ar[rr] && Q_V(\Omega_i) \\
Q_W(M_i) \ar@{^{(}->}[ru]  \ar[rr] && Q_V(M_i) \ar@{^{(}->}[ru]  &}
\end{equation}
in which the vertical maps in the front plane are isomorphisms of glider representations and $\mc{G}$ is a strict sheaf in the sense that the restriction morphisms are strict filtered maps.
\end{definition}
Observe that by \propref{backandforth2} an isomorphism of sheaves between $\mc{F}(T_i -)$ and $\mc{O}_{M_i}$ implies an isomorphism of sheaves between $\mc{G}(T_i -)$ and $\mc{O}_{\Omega_i}$. Hence, the vertical maps in the back plane are strict isomorphisms of filtered $FQ_{WT_i}(R)$-modules.  Serre's global section theorem  for quasi-coherent sheaves of $\mc{O}_R$-modules (see \cite[Theorem 2.1.4]{VoAG}) remains valid in the glider case:
\begin{theorem}
Let $\mc{F} \subset \mc{G}$ be a quasi-coherent sheaf of $\mc{O}_R$ glider representations with $G\mc{F}(W)$ faithful over $G(R)$ for all $W$.  If $\mc{F}(1)$ denotes the global section $F^+R$-glider representation then $\mc{F}$ is sheaf isomorphic to the structure sheaf $\mc{O}_{\mc{F}(1)}$.
\end{theorem}
\begin{proof}
\propref{backandforth2} shows that $\mc{F}$ being a sheaf is equivalent to $\mc{G}$ being a sheaf of filtered $\mc{O}_R$-modules. By the global section theorem for quasi-coherent sheaves of modules (cf. \cite{VoWi}), we have that $\mc{G} \cong \mc{O}_{\mc{G}(1)}$. The latter sheaf is a sheaf of filtered $\mc{O}_R$-modules by the quotient filtration. Denote this filtration by $F\mc{G}$. The original filtration on $\mc{G}$ is denoted by $f\mc{G}$. Since $\mc{G}$ is a strict sheaf $\mc{G}(1) \to \mc{G}(W) = Q_W(\mc{G}(1))$ is strict filtered for every $W \in \un{\mc{W}}(R)$. Let $x \in f_m\mc{G}(W)$, say $W = S_1\ldots S_n$. There are $s_i \in S_i$ such that $s_1\ldots s_n x \in \mc{G}(1)$. Moreover, since $\mc{G}$ is a sheaf of filtered $\mc{O}_R$-modules, $\mc{G}(W)$ is in particular a filtered $FR$-module and therefore $s_1\ldots s_n x \in f_{m+d}\mc{G}(W)$ where $d = \deg(s_1\ldots s_n)$. Hence $s_1\ldots s_n x \in f_{m+d}\mc{G}(W) \cap \mc{G}(1) = f_{m+d}\mc{G}(1)$ by strictness. For $W =1$, the quotient filtration corresponds with the original filtration: $F\mc{O}_{\mc{G}(1)}(1) = F\mc{G}(1) = f\mc{G}(1)$. Hence $s_1\ldots s_n x \in F_{m +d}\mc{G}(1)$ and by definition of the quotient filtration this means that $x \in F_{m}\mc{O}_{\mc{G}(1)}(W)$. This shows that $f_m \mc{G}(W) \subset F_m\mc{G}(W)$. By the faithfulness assumption, multiplying by $s_1\ldots s_n$ does not lower the degree for $f\mc{G}$. Hence the converse $F_m\mc{G}(W) \subset f_m \mc{G}(W)$ is obvious. We conclude that $\mc{G} \cong \mc{O}_{\mc{G}(1)}$ is an isomorphism of filtered $\mc{O}_R$-modules. Reducing to the degree 0 part, yields $\mc{F} \cong \mc{O}_{\mc{F}(1)}$ as desired.
\end{proof}

As we already mentioned in the introduction, we intend to apply these developed techniques a.o. for rings of differential operators (see remarks before \theref{locfilt}) and in classical algebraic geometry. For the latter, we lay out a few examples here.
\begin{example}
Consider the normalization of the cusp $C: X^3 = Y^2$, given by the embedding of coordinate rings 
$$\mathbb{C}[X^2,X^3] \subset \mathbb{C}[X],$$
where the standard filtration from the introduction is considered by choosing $X$ as an $\mathbb{C}[X^2,X^3]$-ring generator for $\mathbb{C}[X]$. For this filtration the associated graded is isomorphic to $\mathbb{C}[X^2,X^3][\epsilon]/(\epsilon^2)$. We have a glider representation 
$$\Omega = \mathbb{C}[X] \supset M = \mathbb{C}[X^2,X^3] \supset \mathbb{C}X^2 + (X^4) \supset (X^5) \supset \ldots$$
and one easily sees that the ideal generated by $\epsilon$ sits in $\Ann(M)$. If we localize the glider at $S = S_X$, then we obtain a glider inside $\mathbb{C}[X,X^{-1}]$ which has non-zero body. Indeed for any $n \geq 2$ and $\gamma \geq 0$ we have that $F_\gamma (X^n)X^4 = (X^{4+n}) \subset F_{-n - 1}\Omega \subset F_{\gamma -n - 1}\Omega$, so $X^4 \in F_{-n-1}(S_X^{-1}\Omega$), so $X^4$ sits in the body.
\end{example}
\begin{example}
Let $K$ be an algebraically closed field. Consider 
$$V = \Spec~ K[X,Y,T]/(XY - T) \quad  {\rm and} \quad W =\Spec~ K[T].$$ The natural morphism between both algebras gives an inclusion $i: W \to V$. If we choose $\ov{X},\ov{Y}$ to be $K[W]$-ring generators for $K[V]$, we have a standard filtration on $K[V]$ with associated graded $G(K[V]) = K[T][\ov{X},\ov{Y}]/(\ov{X}\ov{Y})$. Consider the glider representation 
$$ \Omega = K[V] \supset M = K[T,\ov{Y}] \supset (\ov{Y}) \supset (\ov{Y})^2 \supset \ldots$$\
which is bodyless. Its associated graded is 
$$G(M) = K[T][\ov{Y}],$$
with $\deg(T) = 0, \deg(\ov{Y}) = -1$. We have that $\Ann(M) = \gen(\kappa_{S_{\ov{Y}}})$, the torsion theories that contain $\ov{Y}$. Hence the strong characteristic variety is a proper subset of $V$. On the Zariski open $X(\ov{Y})$, we have that $\wt{M}(X(\ov{Y})) = Q_{\kappa_{S_{\ov{Y}}}}(M)$ has non-zero body. Indeed, it contains $K[\ov{Y}]$. Obviously, we can interchange the role of $X$ and $Y$.
\end{example}
The previous examples indicates a link with singularities or smoothness of morphisms between schemes. Recall that a morphism $f: X \to Y$ of schemes of finite type over some field $k$ is smooth if for every closed point $x \in X$, the induced map on the Zariski tangent spaces $T_f = (df)^*: T_x \to T_{f(x)}$ is surjective. 
\begin{proposition}
Let $K$ be algebraically closed and let $K[W] \subset K[V]$ be an embedding of coordinate rings corresponding to a dominant polynomial map $f: W \to V$, where $W$ and $V$ are embedded in $\mathbb{A}^m$, resp. $\mathbb{A}^n$ with $m < n$. If $f$ is not smooth, there exists a standard filtration on $K[V]$ with degree zero part $K[W]$ and a glider $M \subset K[V]$ with proper strong characteristic variety.
\end{proposition}
\begin{proof}
Let $X_1,\ldots, X_m$ be the variables in $\mathbb{A}^m$ and extend them to $X_1,\ldots,X_n$ for $\mathbb{A}^n$. We may assume that $f$ is not smooth in the origin, so the map 
$$df : \frac{P \cap K[W]}{(P\cap K[W])^2} \to \frac{P}{P^2}$$
is not injective, where $P = (\ov{X_1},\ldots,\ov{X_n})$. W.l.o.g. we may assume that $\ov{X_1} \in P \cap K[W]$ is such that $\ov{X_1} \in P^2 \setminus K[W]$. Hence, we can write $\ov{X_1} = f(\un{X})g(\un{X})$, where $f,g$ are two polynomials with zero constant coefficient. Write $K[V] = K[W][f,g, a_1,\ldots, a_k]$ as a ring extension and use these generators to define a standard filtration on $K[V]$. We have a glider representation
$$K[V] \supset M = K[W][a_1,\ldots,a_k][f] \supset (f) \supset (f)^2 \supset \ldots$$
The associated graded is $G(M) = K[W][a_1,\ldots,a_k][\ov{f}]$, where $\deg(\ov{f}) = -1$. It is clear that the torsion theory $\ov{S}_{\ov{f}} \notin \xi(M)$ and we are done.
\end{proof}

However, the above (standard) filtrations appearing in algebraic geometry are not the only interesting ones. Admittedly, we laid out the sheaf theory for almost commutative rings, such that all filtered localizations were separated and induced the original filtration. This is however not necessary. Indeed, on the one hand we already discussed the existence of the strong characteristic variety related to separatedness of the filtered localization. On the other hand, there are interesting situations where a filtered localization does exist but is not inducing. Of course, the situations we have in mind are the second type of filtrations mentioned in the introduction and hinted at in \exaref{coordinate}. We end the paper by briefly looking at this second type of filtrations.\\

Let $R$ be a ring with given $\mathbb{Z}$-filtration 
$$ \ldots \subset 0 \subset 0 \subset S = F_0R \subset R_1 = F_1R \subset \ldots \subset R = F_nR \subset R \subset R \subset \ldots$$
where each $R_i$ is a subring of $R$. Consider a kernel functor $\kappa$ on $R$-mod and write $\pi: R \to \ov{R} = R/\kappa(R)$ for the canonical epimorphism, putting $F_i\ov{R} = \pi(R_i)$. Since the associated graded $G(R)$ is certainly not a domain, we can't hope for $\kappa$-separatedness. Nonetheless, we can still define a filtration on the localization
$$F_dQ_\kappa(R) = \{ x \in Q_\kappa(R), \exists I \in \mc{L}(\kappa) {\rm~such~that~} F_\gamma I x \subset F_{\gamma + d} \ov{R}, ~\forall \gamma \in \mathbb{Z}\}.$$
\begin{proposition}\proplabel{toren}
With notations as above: $FQ_\kappa(R)$ is a separated (exhaustive) filtration of the ring $Q_\kappa(R)$ making the canonical $j_\kappa: R \to Q_\kappa(R)$ into a filtered morphism, which need not be strict.
\end{proposition}
\begin{proof}
We only have to show that for $y \in F_pQ_\kappa(R), z \in F_qQ_\kappa(R)$, we have that $yz \in F_{p+q}Q_\kappa(R)$, the other properties of an exhaustive filtration are easily checked. Say $I \in \mc{L}(\kappa)$ is such that $F_\gamma I y \subset F_{\gamma + p}\ov{R}$ for all $\gamma \in \mathbb{Z}$, $J \in \mc{L}(\kappa)$ is such that $F_\mu J z \subset F_{\mu + q}\ov{R}$ for all $\mu \in \mathbb{Z}$. Let $I_1 \in \mc{L}(\kappa)$ be the set $\{ x \in R, ~xy \in \ov{R}\}$. Now look at $(\ov{J}:y) = \{ x \in R,~xy \in \ov{J}\}$, then $(\ov{J}:y) \subset I_1$ and for every $i \in I_1$ we have that $(J:iy) = \{r \in R,~ riy \in \ov{J}\} \in \mc{L}(\kappa)$. We also have that $(J:iy)iy \subset \ov{J}$ hence $(J:iy)i \subset (\ov{J}:y)$, or $(J:iy) \subset ((\ov{J}:y):i)$. By idempotency of $\kappa$ it follows that $(\ov{J}:y) \in \mc{L}(\kappa)$ and hence $H = (\ov{J}:y) \cap I \in \mc{L}(\kappa)$. $F_\gamma H \subset F_\gamma I$ for all $\gamma$ thus $F_\gamma H y \subset F_{\gamma + p} \ov{J}$ and thus 
$$F_\gamma H yz \subset F_{\gamma + p} \ov{J} z \subset F_{\gamma + p + q} \ov{R}.$$
In other words $yz \in F_{p+q}Q_\kappa(R)$.\\
It remains to verify that $FQ_\kappa(R)$ is separated, so assume that $z \in \cap_d F_dQ_\kappa(R)$. This means that for any $d$ there is an $I \in \mc{L}(\kappa)$ such that $F_\gamma I z \subset F_{\gamma +d} \ov{R}$, in particular for $d < - n$. We have for $\gamma =n$ that $I z \subset F_{ n +d} \ov{R} \subset F_{-1}\ov{R} = 0$ with $I \in \mc{L}(\kappa)$. But $z \in Q_\kappa(R)$ then entails that $z =0$ and $FQ_\kappa(R)$ is indeed separated.
\end{proof}
For the localization morphism $R \to Q_\kappa(R)$ to be strict we would have to have that $\ov{R} \cap F_nQ_\kappa(R) = F_n \ov{R}$; however $F_\gamma I x \subset F_{\gamma + n} \ov{R}$ is possible for $x \notin F_n\ov{R}$. For example, take $ x \in F_{n+1}\ov{R} - F_n\ov{R}$ (such $n$ exists by separability) with $\max\{\gamma,n+1\} < \gamma +n$, then \newline $F_\gamma I x \subset F_{\max\{ \gamma, n +1\}}\ov{R} \subset F_{\gamma + n}\ov{R}$ with $x \notin F_n \ov{R}$.\\
In the above situation we now consider an $R$-module $M$ with discrete filtration of finite length $m+n$ say, starting at $-m$ and reaching $M$ at $F_nM = M$. Then $\ov{M} = M /\kappa(M)$ is discrete too and hence $F\ov{M}$ is separated as the induced filtration by $FM$. Define
$$F_dQ_\kappa(M) = \{ q \in Q_\kappa(M), \exists I \in \mc{L}(\kappa) {\rm~such~that~} F_{\gamma}I q \subset F_{\gamma + d} \ov{M}\}.$$
\begin{proposition}
With notation as above: $FQ_\kappa(M)$ is a separated (exhaustive) filtration and $Q_\kappa(M)$ is a filtered $Q_\kappa(R)$=module with respect to the quotient filtration on $Q_\kappa(R)$ as in \propref{toren}.
\end{proposition}
\begin{proof}
We just point out the separability, the other properties are easily checked just as in \propref{toren}. So look at $q \in \cap_d F_dQ_\kappa(M)$. For $d = -n -m -1$, $\forall \gamma,~ F_\gamma I q \subset F_{q - n -m -1}\ov{M}$ for some $I \in \mc{L}(\kappa)$. For $\gamma =n$, we obtain that $I q \subset F_{-m-1}\ov{M} = 0$, hence $q = 0$.
\end{proof}
\begin{remark} We actually prove in the above propositions that $Q_\kappa(R), Q_\kappa(M)$ are discrete filtered ($F_{-n-1}Q_\kappa(R) = 0, F_{-n-m-1}Q_\kappa(M) = 0$).
\end{remark}
Again looking at a finite ring-filtered ring $R$ we now consider a glider representation $M \subset \Omega$ where $M$ is a finite length fragment, say $M_m =0, M_{m-1} \neq 0$. Then the filtration we introduced on $RM$ starts at $F_{-m}RM = 0, F_{-m+1}RM = M_{m-1}$ and reaches $RM$ at $F_nRM$, so it is discrete of length $m+n$. Put $\Omega = RM$ and consider $Q_\kappa(\Omega)$ with the quotient filtration $FQ_\kappa(\Omega)$ as defined above and we define $Q_\kappa(M) = F_0Q_\kappa(\Omega)$ as earlier in the paper. Then $Q_\kappa(M)$ is a finite glider representation for $Q_\kappa(R)$ in $Q_\kappa(\Omega)$.\\

Now we specify the nature of $\kappa$ with respect to $S = F_0R$. We assume that $\kappa$ induces a localization $\kappa_S$ on $S$ via $I_1 \in \mc{L}(\kappa_S)$ if $I_1 = I \cap S$ for some $I \in \mc{L}(\kappa)$ and for any $J \in \mc{L}(\kappa),~R(J\cap S) \in \mc{L}(\kappa)$. We also assume that $S$ is (left) Noetherian (in particular this applies to a tower of coordinate rings $K[W] \subset \ldots \subset K[V]$), then $\kappa_S$ and also $\kappa$ are of finite type, that is, $\mc{L}(\kappa_S)$ and $\mc{L}(\kappa)$ have a filterbasis of finitely generated left ideals. We say that the glider $M \subset \Omega$ is $(\kappa,\kappa_S)$-orthogonal if $\Omega/F_d\Omega$ is $\kappa_S$-torsion free for all $d$.
\begin{proposition}
If $M$ is a finite length glider over the finite ring-filtered $R$, $M \subset \Omega =RM$ and $\kappa, \kappa_S$ are symmetric kernel functors on $R$-mod, resp. $S$-mod, such that $M$ is $(\kappa,\kappa_S)$-orthogonal, then the separated filtration on $Q_\kappa(\Omega)$ with $F_0Q_\kappa(M) = Q_\kappa(M)$ induces the filtration $F\ov{\Omega}$ on $\ov{\Omega}$. Similarly, if $R$ is $(\kappa,\kappa_S)$-orthogonal then $FQ_\kappa(R)$ induces $F\ov{R}$, i.e. $j: R \to Q_\kappa(R)$ is strict filtered.
\end{proposition}
\begin{proof}
We write the proof for $M, \Omega$, the proof for $R$ is similar. Take $\ov{q} \in F_dQ_\kappa(\Omega) \cap \ov{\Omega}$, i.e. there is an $I \in \mc{L}(\kappa)$ such that $F_\gamma I \ov{q} \subset F_{\gamma + d}\ov{\Omega}$ for all $\gamma$. In particular for $\gamma = 0$, we have that $(I \cap S)\ov{q} \subset F_d\ov{\Omega}$ with $I \cap S \in \mc{L}(\kappa_S)$. Pick a representative $q \in \Omega$ for $\ov{q}$, then $(I \cap S)q \subset F_d\Omega + \kappa_S(\Omega)$. Since $S$ is Noetherian, $I \cap S$ is $\sum_{j=1}^tSi_j$ for some finite $i_j \in I$. Then $i_jq \in F_d\Omega + \kappa(\Omega)$, say $i_jq = f_{d,j} + t_j$. Take $J_i \in \mc{L}(\kappa)$ such that $J_it_i = 0$, then $(S \cap J_i)i_jq \subset F_d \Omega$ and it follows that $(\sum (S \cap J_i)i_j)q \subset F_d\Omega$. Since $\kappa_S$ is symmetric, $\sum(S \cap J_i)i_j in \mc{L}(\kappa_S)$. So if $\Omega/F_d\Omega$ is $\kappa_S$-torsion free, $q \in F_d\Omega$ and $\ov{q} \in F_d\ov{\Omega}$. Thus $\ov{\Omega} \to Q_\kappa(\Omega)$ is strict and $j_\Omega \to Q_\kappa(\Omega)$ is strict too. In particular, the fragment structure of $F_0Q_\kappa(\Omega) =Q_\kappa(M)$ induces the chain of the fragment $M$.
\end{proof}
\begin{corollary}
Let $K[W] \subset \ldots \subset K[V_i] \subset \ldots \subset K[V] = K[V_n]$ be a tower of coordinate rings for a chain of varieties $V \to \cdots \to V_i \to \cdots \to W$ and $K(V_i) \cap K[V] = K[V_i]$, say all $K[V_i]$ are Noetherian domains, then every $\kappa$ inducing $\kappa_W$ as before is such that $K[V]$ is $(\kappa,\kappa_W)$-orthogonal, hence the discrete quotient filtration on $Q_\kappa(K[V])$ induces the finite ring filtration on $K[V]$.
\end{corollary}
\begin{proof}
Let $z \in K[V]$ be such that $(I \cap K[W])z \subset K[V_i]$ for some $i < n$. Then for some $a \in I \cap K[W]$ we have that $az \in K[V_i]$ or that $z \in K(V_i)$. Hence $z \in K(V_i) \cap K[V] = K[V_i]$.
\end{proof}
\begin{remark} 
The condition concerning the function fields of the varieties rules out any birationalities in the chain of varieties. 
\end{remark}
\section*{Acknowledgement}
Part of this research was done while the first author visited the Department of Mathematics at Fudan University.
The first author thanks Quanshui Wu for his warm hospitality.

\end{document}